\documentclass[centertags,12pt]{amsart}
\usepackage{amssymb}
\usepackage{verbatim}
\usepackage{array}
\usepackage{latexsym}
\usepackage{enumerate}
\usepackage{amsmath}
\usepackage{amsfonts}
\usepackage{amsthm}
\usepackage{color}
\usepackage[english]{babel}

\usepackage{tikz}
\usepackage{wasysym}
\usepackage[foot]{amsaddr}

\usepackage{mathrsfs}
\usepackage{setspace}
\usepackage{ textcomp }
\usepackage{graphicx}
\usepackage{cite}
\usepackage{comment,cancel}
\graphicspath{ {images/} }
\usepackage{listings}

\newtheorem{theorem}{Theorem}[section]
\newtheorem{proposition}[theorem]{Proposition}
\newtheorem{lemma}[theorem]{Lemma}
\newtheorem{corollary}[theorem]{Corollary}

\def\cH{\mathcal H}

\def\cO{\mathcal O}

\def\cX{\mathcal X}
\def\cY{\mathcal Y}

\def\Aut{\mbox{\rm Aut}}
\def\K{\mathbb{K}}

\def\deg{\mbox{\rm deg}}

\def\Aut{\mbox{\rm Aut}}

\newcommand{\aut}{\mbox{\rm Aut}}

\newcommand{\KK}{K}

\def\neg1{\text{\boldmath$1$}}

\def\neg1{\text{\boldmath$1$}}

\def\cH{\mathcal H}

\def\cX{\mathcal X}
\def\cY{\mathcal Y}

\def\KK{\mathbb{K}}

\def\FF{\mathbb{F}}

\newtheorem{remark}[theorem]{Remark}

\begin{document}

\author[M. Montanucci]{Maria Montanucci}
\address{Department of Applied Mathematics and Computer Science, Technical University of Denmark, 2800 Kongens Lyngby, Denmark}
\email{marimo@dtu.dk}
\author[G. Tizziotti]{Guilherme Tizziotti}
\address{Faculdade de Matemática, Universidade Federal de Uberlândia, 2121 Uberlândia, Brazil}
\email{guilhermect@ufu.br}
\author[G. Zini]{Giovanni Zini}
\address{Department of Physics, Informatics, and Mathematics, University of Modena and Reggio Emilia, 41125 Modena, Italy}
\email{giovanni.zini@unimore.it}

\title[The automorphism group of a family of maximal curves]{On the automorphism group of a family of maximal curves not covered by the Hermitian curve}

\maketitle
\begin{abstract}
In this paper we compute the automorphism group of the curves $\mathcal{X}_{a,b,n,s}$ and $\mathcal{Y}_{n,s}$ introduced in Tafazolian et al. \cite{TTT} as new examples of maximal curves which cannot be covered by the Hermitian curve. They arise as subcovers of the first generalized GK curve (GGS curve).
As a result, a new characterization of the GK curve, as a member of this family, is obtained.
\end{abstract}

{\bf Keywords:} maximal curve, GK curve, automorphism group.

{\bf MSC 2020:} 14H37, 11G20

\section{Introduction}

Let $\cX$ be a nonsingular, projective, geometrically irreducible algebraic curve of positive
genus g defined over a finite field $\mathbb{F}_q$ with $q$ elements and let $\cX(\mathbb{F}_q)$ be the set of its
$\mathbb{F}_q$-rational points. The curve $\cX$ is called $\mathbb{F}_q$-maximal if its number of $\mathbb{F}_q$-rational point
attains the Hasse-Weil upper bound, namely equals $2g\sqrt{q} + q + 1$. 
Clearly, maximal curves can only exist over fields whose cardinality is a perfect square. Apart for being
of theoretical interest as extremal objects, maximal curves over finite fields have attracted
a lot of attention in recent decades due to their applications to coding theory and cryptography. Maximal curves are indeed special for the structure of the so-called Weiestrass semigroup at one point, which is the main ingredient used in the literature to construct AG codes with good parameters.

The most important and well-studied example of a maximal curve is the so-called Hermitian curve $\mathcal{H}_q$ defined over $\mathbb{F}_{q^2}$ by the affine equation 

$$ \mathcal{H}_q: X^q + X = Y^{q+1}.$$ 
A well-known reason is that for fixed $q$, the curve $\mathcal{H}_q$ has the largest possible genus $g(\mathcal{H}_q) = q(q - 1)/2$ that an
$\mathbb{F}_{q^2}$-maximal curve can have. A result commonly attributed to Serre, see \cite[Proposition 6]{Serre}, gives that any curve which is $\mathbb{F}_{q^2}$-covered by an $\mathbb{F}_{q^2}$-maximal curve is $\mathbb{F}_{q^2}$-maximal. 

Therefore many maximal curves can be obtained by constructing subcovers of already known maximal curves, in particular subcovers of the Hermitian curve. For
a while it was speculated in the research community that perhaps all maximal curves could be obtained as subcovers of the Hermitian curve, but it was shown by Giulietti and Korchmáros that this is not the case, see \cite{GK}. 

Giulietti and Korchmáros constructed indeed a maximal curve over $\mathbb{F}_{q^6}$, today referred to as GK curve, which cannot be
covered by the Hermitian curve whenever $q$ is larger than $2$. Garcia, G\"uneri, and Stichtenoth, in \cite{GGS},
presented a new family of maximal curves over $\mathbb{F}_{q^{2n}}$ (where $n$ is odd), known as GGS curves, which
generalizes the GK curve (when $n=3$ the GGS curve and the GK curve coincide) and that is not Galois-covered by the Hermitian curve \cite{GGS,GMZ}.
Many applications of these curves in coding theory have been made in recent years, see e.g. \cite{zini}, \cite{BMZ}, \cite{CT_GK}, \cite{GKcodes}, \cite{HY3} and \cite{CT mGK}.

Another generalization of the GK curve over $\mathbb{F}_{q^{2n}}$ (again $n$ odd) has been introduced by Beelen and Montanucci in \cite{BM}, which is now known as BM curves. These curves are not Galois-covered by the Hermitian curve as well, unless $q=2$. Applications
of the BM curves to coding theory can be found in \cite{LV} and \cite{MP}.

Tafazolian, Teherán-Herrera, and Torres \cite{TTT} presented two further examples of maximal
curves over $\mathbb{F}_{q^{2n}}$ ($n$ odd), denoted by $\cX_{a,b,n,s}$ and $\cY_{n,s}$, that cannot be covered by the Hermitian curve. These
examples are again closely related to the GK curve, as $\cY_{3,1}$ is exactly the GK constructed in \cite{GK}.
They can be further seen as generalizations of the GGS, as $\cY_{n,1}$ is the GGS curve corresponding to the same parameter $n$. 

The aim of this paper is to compute the full automorphism groups of the curves $\cX_{a,b,n,s}$ and $\cY_{n,s}$ over the algebraic closure of $\mathbb{F}_{q^{2n}}$.
More precisely, the following are the two central results obtained in this paper (the precise definition of the subgroups involved in the statement are in Sections \ref{sec:prel} and \ref{sec:results}).

\begin{theorem}\label{th:1} 
Let $q$ be a prime power, $n \geq 3$ odd, $m:=(q^n+1)/(q+1)$, and $s$ a divisor of $m$ with $s\ne m$.
If $3\nmid n$ or $\frac{m}{s}\nmid(q^2-q+1)$, then ${\rm Aut}(\mathcal{Y}_{n,s})$ has order $q^3(q^2-1)m/s$ and is isomorphic to $S_{q^3}\rtimes C_{(q^2-1)\frac{m}{s}}$.
    If $3\mid n$ and $\frac{m}{s}\mid (q^2-q+1)$, then ${\rm Aut}(\mathcal{Y}_{n,s})$ has order $(q^3+1)q^3(q^2-1)m/s$ and is isomorphic to ${\rm PGU}(3,q)\rtimes C_{m/s}$. 
\end{theorem}

\begin{theorem}\label{th:2}
Let $q=p^a$ be a prime power, $n \geq 3$ odd, $m:=(q^n+1)/(q+1)$, $s$ a divisor of $m$, and $b$ a divisor of $a$. Assume that $b<a$ or $q^2\nmid(\frac{m}{s}-1)$. Then the automorphism group of $\cX_{a,b,n,s}$ has order $\frac{q^3}{\bar{q}}(q+1)(\bar{q}-1)\frac{m}{s}$ and is isomorphic to $(S_{q^3}/E_{\bar{q}})\rtimes C_{(q+1)(\bar{q}-1)m/s}$. 
\end{theorem}

Theorems \ref{th:1} and \ref{th:2} provide a new characterization of the GK curve as a member of the family of maximal curves $\mathcal{Y}_{n,s}$, $\mathcal{X}_{a,b,n,s}$ given in \cite{TTT}. 
Indeed, Theorems \ref{th:1} and \ref{th:2} show which members in that family admit an automorphism group isomorphic to ${\rm PGU}(3,q)$, i.e. when the full automorphism group of the underlying Hermitian curve $\mathcal{H}_q$ can be completely lifted: they are exactly the GK curve $\mathcal{GK}$, together with its quotients $\mathcal{GK}/C$ over a subgroup $C$ of the Galois group $C_{q^2-q+1}$ of $\mathcal{GK}\to\mathcal{H}_q$.
Theorem \ref{th:1} also provides a different proof of the structure of the automorphism group of the GGS curve with respect to the ones given in \cite{GOS} and \cite{GMP}.

The paper is organized as follows. In Chapter 2 the necessary background on maximal curves, their automorphism groups and general results in group theory are recalled. Chapter 3 contains the proofs of the two aforementioned theorems.

\section{Preliminary results}\label{sec:prel}

\subsection{Automorphism groups of algebraic curves}

In this paper, $\cX$ stands for a (projective, geometrically irreducible, non-singular) algebraic curve of genus $g=g(\cX) \ge 2$ defined over an algebraically closed field $\K$ of positive characteristic $p$. Let $\aut(\cX)$ be the group of all automorphisms of $\cX$. The assumption $g(\cX)\geq 2$ ensures that $\aut(\cX)$ is finite. However the classical Hurwitz bound
 $|\aut(\cX)| \leq 84(g(\cX)-1)$ for complex curves fails in positive characteristic, and there exist four families of curves satisfying $|\aut(\cX)|\geq 8g(\cX)^3$; see \cite{stichtenoth1973II}, \cite{henn1978}, and \cite[Section 11.12]{HKT}. 

For a subgroup $G$ of $\aut(\cX)$, let $\bar \cX$ denote a non-singular model of $\K(\cX)^G$, that is,
a (projective, geometrically irreducible, non-singular) algebraic curve with function field $\K(\cX)^G$, where $\K(\cX)^G$ is the fixed field of $G$, i.e. the subfield of $\K(\cX)$ fixed elementwise by every element in $G$. Usually, $\bar \cX$ is called the
quotient curve of $\cX$ by $G$ and denoted by $\cX/G$. The field extension $\K(\cX)/\K(\cX)^G$ is Galois of degree $|G|$.

Let $\Phi$ be the natural covering $\cX \rightarrow \bar{\cX}$, where $\bar{\cX}=\cX/G$. A point $P\in\cX$ is a ramification point of $G$ if the stabilizer $G_P$ of $P$ in $G$ is nontrivial; the ramification index $e_P$ is $|G_P|$; a point $\bar{Q}\in\bar{\cX}$ is a branch point of $G$ if there is a ramification point $P\in \cX$ such that $\Phi(P)=\bar{Q}$; the ramification (branch) locus of $G$ is the set of all ramification (branch) points. The $G$-orbit of $P\in \cX$ is the subset
$o=\{g(P)\in\mathcal{X}\colon g\in G\}$ of $\cX$, and it is {\em long} if $|o|=|G|$, otherwise $o$ is {\em short}. For a point $\bar{Q}\in\bar{\cX}$, the $G$-orbit $o$ lying over $\bar{Q}$ consists of all points $P\in\cX$ such that $\Phi(P)=\bar{Q}$. If $P\in o$ then $|o|=|G|/|G_P|$ and hence $\bar{Q}$ is a branch point if and only if $o$ is a short $G$-orbit. It may be that $G$ has no short orbits. This is the case if and only if every non-trivial element in $G$ is fixed--point-free on $\cX$, that is, the covering $\Phi$ is unramified. On the other hand, $G$ has a finite number of short orbits.

For a non-negative integer $i$, the $i$-th ramification group of $\cX$
at $P$ is denoted by $G_P^{(i)}$ (or $G_i(P)$ as in \cite[Chapter
IV]{serre1979})  and defined to be
$$G_P^{(i)}=\{\alpha \in G_P \;\colon\; v_P(\alpha(t)-t)\geq i+1\}, $$ where $t$ is a uniformizing element (local parameter) at
$P$. 
The main properties of the subgroup chain $G_P^{(0)} \supseteq G_P^{(1)} \supseteq G_P^{(2)} \supseteq \ldots$ are collected in the following lemma.
\begin{lemma}{\rm \cite[Theorem 11.49 and Theorem 11.74]{HKT}}  \label{highnormal}
\begin{enumerate}
\item $G_P^{(0)}=G_P=S \rtimes H$, where $S$ is a $p$-group and $H$ is a cyclic group of order not divisible by $p$.
\item $G_P^{(1)}=S$ is the unique Sylow $p$-subgroup (and maximal normal subgroup) of $G_P$.
\item For every $i \geq 1$, $G_P^{(i)}$ is normal in $G_P$ and the quotient group $G_P^{(i+1)}/G_P^{(i)}$ is elementary abelian.
\end{enumerate}
\end{lemma}

Let $\bar{g}$ be the genus of the quotient curve $\cX/G$. The \emph{Hurwitz
genus formula} (also called Riemann-Hurwitz formula, see \cite[Theorem 3.4.13]{stichtenoth}) gives the following equation:
    \begin{equation}
    \label{eq1}
2g-2=|G|(2\bar{g}-2)+\sum_{P\in \cX} d_P,
    \end{equation}
    where the different $d_P$ at $P$ is given by
\begin{equation}
\label{eq1bis}
d_P= \sum_{i\geq 0}(|G_P^{(i)}|-1),
\end{equation}
see \cite[Theorem 11.70]{HKT}. Clearly the above contribution $\sum_{P\in \cX} \sum_{i\geq 0}(|G_P^{(i)}|-1)$ can be re-written by summing with respect to the elements of $ \alpha \in G$ and counting the number of $P's$ and $i's$ such that $\alpha \in G_P^{(i)}$. Doing so one can re-write the formula above as
    \begin{equation}
    \label{eqi}
2g-2=|G|(2\bar{g}-2)+\sum_{\alpha \in G \setminus \{id\}} i(\alpha),
    \end{equation}
where $i(\alpha)$ is called the contribution of the automorphism $\alpha$ to the covering $\Phi$.

Let $\gamma$ be the $p$-rank of $\cX$, and let $\bar{\gamma}$ be the $p$-rank of the quotient curve $\cX/G$.
A formula relating $\gamma$ and $\bar \gamma$ is known whenever $G$ is a $p$-group: in this case, the \emph{Deuring-Shafarevich formula} states that
\begin{equation}
    \label{eq2deuring}
\gamma-1={|G|}(\bar{\gamma}-1)+\sum_{i=1}^k (|G|-\ell_i),
    \end{equation}
where $\ell_1,\ldots,\ell_k$ are the sizes of the short orbits of $G$; see \cite{sullivan1975} or \cite[Theorem 11.62]{HKT}.

A subgroup of $\aut(\cX)$ is  a prime-to-$p$ group, or a $p'$-group, if its order is prime to $p$. A subgroup $G$ of $\aut(\cX)$ is {\em{tame}} if the $1$-point stabilizer in $G$ of any point of $\cX$ is a $p'$-group. Otherwise, $G$ is {\em{non-tame}} (or {\em{wild}}). By \cite[Theorem 11.56]{HKT}, if $|G|>84(g(\cX)-1)$ then $G$ is non-tame. 
An orbit $o$ of $G$ is {\em{tame}} if $G_P$ is a $p'$-group for every $P\in o$. 

The following lemma gives a strong restriction to the action of the Sylow $p$-subgroup of the $1$-point stabilizer when $\cX$ is a maximal curve. Actually, it holds for the class of curves with $p$-rank $\gamma=0$, which contains the maximal curves, see e.g. \cite[Theorem 9.76]{HKT}.
 
 \begin{lemma}{\rm \cite[Proposition 3.8, Theorem 3.10]{GSY}} \label{prankpel} Let $\mathcal{X}$ be an 
$\mathbb{F}_{q^2}$-maximal curve of genus $g\geq2$. Then the automorphism group $\aut(\mathcal{X})$ fixes the set 
$\mathcal{X}(\mathbb{F}_{q^2})$ of $\mathbb{F}_{q^2}$-rational points$.$ Also, automorphisms of $\mathcal{X}$ over the algebraic 
closure of $\mathbb{F}_{q^2}$ are always defined over $\mathbb{F}_{q^2}.$ 
   \end{lemma}

We can use Lemma \ref{prankpel} to ensure that a Sylow $p$-subgroup of a non-tame 
automorphism group of an $\mathbb{F}_{q^2}$-maximal curve $\mathcal{X}$ fixes exactly one 
$\mathbb{F}_{q^2}$-rational point of $\mathcal{X}$.

\begin{corollary}\label{1actionp2} Let $p$ denote the characteristic of the finite field $\mathbb{F}_{q^2}$ where $q=p^t$ and let $\mathcal{X}$ be an $\mathbb{F}_{q^2}$-maximal curve with genus $g=g(\mathcal{X}) \geq 2$ such that $p \mid |\aut(\mathcal{X})|$. If $H$ is a $p$-subgroup of 
$\aut(\mathcal{X}),$ then $H$ fixes exactly one point $P\in \mathcal{X}(\mathbb{F}_{q^2})$ and acts semiregularly on the 
set of the remaining $\mathbb{F}_{q^2}$-rational points of $\mathcal{X}$.
  \end{corollary}
  
  \begin{proof} Assume first that $H$ is a Sylow $p$-subgroup of $\aut(\mathcal{X})$. Then from Lemma \ref{prankpel}, $H$ acts on the set $\mathcal{X}(\mathbb{F}_{q^2})$ of $\mathbb{F}_{q^2}$-rational 
points of $\mathcal{X}$. Since $|\mathcal{X}(\mathbb{F})| \equiv 1 \pmod{p}$, $H$ must fix at least one 
point $P \in \cX(\mathbb{F}_{q^2})$. Also, $\cX$ has zero $p$-rank and hence the claim 
follows from \cite[Lemma 11.129]{HKT}. 
Now assume that $H$ is an arbitrary $p$-subgroup of $\aut(\mathcal{X})$. Since $H$ is contained in at least one Sylow $p$-subgroup of $\aut(\mathcal{X})$, and the property of fixing a point $P \in \mathcal{X}$ and being semiregular elsewhere is preserved by subgroups, the claim follows for $H$ as well.
   \end{proof}

If the bound $|G|>84(g(\mathcal{X})-1)$ is satisfied, a lot can be said about the structure of the short orbits of $G$ on $\mathcal{X}$. The following theorem lists all the possibilities.

\begin{theorem} \rm{\cite[Theorem 11.56 and Lemma 11.111]{HKT}} \label{Hurwitz}
 Let $\cX$ be an irreducible curve of genus $g \geq  2$ defined over an algebraically closed field $\K$ of characteristic $p$. 
\begin{itemize}
\item If $G$ has at least five short orbits then $|G| \leq 4(g-1)$. 
\item If $G$ has four short orbits then $|G| \leq 12(g-1)$.
\item If $G$ has exactly one short orbit, then the length of this orbit divides $2g-2$.
\item If $p>0$ and $|G|>84(g-1)$ then the fixed field $\K(\cX)^G$ is rational and $G$ has at most three short orbits, namely: 
\begin{enumerate}
\item exactly three short orbits, two tame of length $|G|/2$ and one non-tame, with $p \geq 3$; or
\item  exactly two short orbits, both non-tame; or 
\item  only one short orbit, which is non-tame, whose length divides $2g-2$; or
\item  exactly two short orbits, one tame and one non-tame.
\end{enumerate}
\item In any case $|G|<8g^3$ unless one of the following cases occurs up to isomorphism over $\K$:
\begin{itemize}
\item $p = 2$ and $\cX$ is a non-singular model of $Y^2 + Y = X^{2k+1}$, with $k > 1$; 
\item $p > 2$ and $\cX$ is a non-singular model of $Y^2 = X^n -X$, where $n = p^h$ and $h > 0$;
\item $\cX$ is the Hermitian curve $\mathcal{H}_q: Y^{q+1}=X^q+X$ where $q = p^h$ and $h > 0$;
\item  $\cX$ is the non-singular model of the Suzuki curve $\mathcal{S}_q: X^{q_0}(X^q + X) = Y^q + Y$, where $q_0 = 2^r$, $r \geq 1$ and $q = 2q_0^2$. 
\end{itemize}
\end{itemize}
\end{theorem}

A tool we will use to compute automorphism groups is the so-called Weierstrass semigroup $H(P)$ at a point $P \in \mathcal{X}$:
$$H(P):=\{i \in \mathbb{N} \colon \exists f \in \mathbb{K}(\mathcal{X}), (f)_\infty=iP \}.$$

It is well-known that the set $H(P)$ is a numerical semigroup and that from the Weierstrass gap theorem the set of gaps $G(P):=\mathbb{N} \setminus H(P)$ has cardinality $g$, see e.g. \cite[Theorem 1.6.8]{stichtenoth}.
Points that are in the same orbit under the action of an automorphism group of $\mathcal{X}$ have the same Weierstrass semigroup, see \cite[Lemma 3.5.2]{stichtenoth}.
The following lemma provides a tool to compute gaps.

\begin{lemma}{\rm \cite[Corollary 14.2]{VillaSalvador}} \label{gapsdiff}
 Let $\mathcal X$ be an algebraic curve of genus $g$ defined over $\mathbb{K}$. Let $P$ be a
point of $\mathcal X$ and $\omega$ be a regular differential on $\mathcal X$ . Then $v_P (\omega) + 1$ is a gap at $P$.
\end{lemma}

In the following section the first generalized GK curve will be introduced, and with that also the protagonists of this paper, namely the curves $\mathcal{Y}_{n,s}$ and $\mathcal{X}_{a,b,n,s}$.

\subsection{The first generalized GK curve $\mathcal{C}_n$ (GGS curve)}\label{sec:GGS}

As mentioned in the introduction, from a result commonly attributed to Serre \cite{KS}, we know that every curve which is $\mathbb{F}_{q^2}$-covered by an $\mathbb{F}_{q^2}$-maximal curve is itself also $\mathbb{F}_{q^2}$-maximal. The most important example of $\mathbb{F}_{q^2}$-maximal curve is the Hermitian curve $\mathcal{H}_q$, with affine equation $Y^{q+1}=X^{q+1}-1$ or $Y^{q+1}=X^{q}+X$. The automorphisms group of $\mathcal{H}_q$ is very large compared to $g(\mathcal{H}_q)$. Indeed it is isomorphic to ${\rm PGU}(3,q)$ and its order is larger than $16 g(\mathcal{H}_q)^4$. Moreover $\mathcal{H}_q$ has the largest genus admissible for an $\mathbb{F}_{q^2}$-maximal curve and it is the unique curve having this property up to birational isomorphism, see \cite{RS}.

Few examples of maximal curves not covered by $\mathcal{H}_q$ are known in the literature. In \cite{GK} Giulietti and Korchm\'aros constructed an $\mathbb{F}_{q^6}$-maximal curve, nowadays known as the GK curve, which is not a subcover of the Hermitian curve $\mathcal{H}_{q^3}$ whenever $q\geq3$.
An affine space model for it is
\[
\mathcal{GK}:\begin{cases}
    Z^{q^2-q+1}=Y\frac{X^{q^2}-X}{X^q+X}\\
    Y^{q+1}=X^q+X\\
\end{cases}.
\]
The full automorphism group of $\mathcal{GK}$ has order $(q^3+1)q^3(q^2-1)(q^2-q+1)$. It is generated by two normal subgroups, one isomorphic to ${\rm PGU}(3,q)$ and the other cyclic of order $q^2-q+1$, and contains ${\rm PGU}(3,q)\times C_{\frac{q^2-q+1}{\gcd(3,q+1)}}$ as a normal subgroup of index $\gcd(3,q+1)$; see \cite[Section 5]{GK}.

Two generalizations of the GK curve into infinite families of maximal curves are known in the literature and they are not Galois subcovers of the corresponding maximal Hermitian curve. 
The first generalization $\mathcal{C}_{n}$ was introduced by Garcia, G\"{u}neri and Stichtenoth in \cite{GGS}, whence the name of GGS curve. For any prime power $q$ and odd $n\geq 3$, the GGS curve $\mathcal{C}_{n}$ is given by the affine space model
 \begin{equation}\label{GGS}
\mathcal{C}_{n}:\begin{cases}
Z^m=Y\frac{X^{q^2}-X}{X^q+X} \\
Y^{q+1}=X^{q}+X
\end{cases}
\end{equation} where $m:=\frac{q^n+1}{q+1}$; 
$\mathcal{C}_n$ is equivalently defined by the equations
\[\mathcal{C}_n\colon\quad Z^m=Y^{q^2}-Y,\quad Y^{q+1}=X^q+X.\]
Notice that $\mathcal{C}_{3}$ is the GK curve $\mathcal{GK}$.
The curve $\mathcal{C}_{n}$ is $\mathbb{F}_{q^{2n}}$-maximal of genus $g(\mathcal{C}_{n})=(q-1)(q^{n+1}+q^n-q^2)/2$. 
Whenever $n\geq5$, $\mathcal{C}_n$ is not a Galois subcover of $\mathcal{H}_{q^n}$; see \cite{DM} for $q\geq3$ and \cite{GMZ} for $q=2$.

For any $n\geq5$, the automorphism group of $\mathcal{C}_n$ was determined independently in \cite{GOS} and \cite{GMP}. It has order $q^3(q^2-1)m$ and it is a semidirect product ${\rm Aut}(\mathcal{C}_n)=S_{q^3}\rtimes \Sigma$, where
\begin{eqnarray*}
S_{q^3}&=&\left\{(x,y,z)\mapsto(x+b^q y+c,y+b,z)\;\colon\; b,c\in\mathbb{F}_{q^2},c^q+c=b^{q+1}\right\}\cong E_q\,.\,E_{q^2}\;,\\
\Sigma&=&\left\{(x,y,z)\mapsto(x,y,z)\mapsto(\zeta^{q^n+1}x,\zeta^{m}y,\zeta z)\;\colon\; \zeta^{(q^n+1)(q-1)}=1\right\}\cong C_{(q^n+1)(q-1)}\;.
\end{eqnarray*}
Let $x,y,z$ be the coordinate functions. The function field $\mathbb{F}_{q^{2n}}(x,y,z)$ of $\mathcal{C}_n$ is a Kummer extension of degree $m$ of the Hermitian function field $\mathbb{F}_{q^{2n}}(x,y)$, the Galois group of $\mathbb{F}_{q^{2n}}(x,y,z)/\mathbb{F}_{q^{2n}}(x,y)$ being the subgroup of order $m$ in $\Sigma$. In this extension, the places centered at the $q^3+1$ $\mathbb{F}_{q^2}$-rational points of $\mathcal{H}_q$ are the unique ramified places, and they are totally ramified.

The group ${\rm Aut}(\mathcal{C}_n)$ has exactly two short orbits on $\mathcal{C}_n$, namely the singleton $\{P_{\infty}\}$, where $P_{\infty}$ is an $\mathbb{F}_{q^2}$-rational point of $\mathcal{C}_n$ (common pole of $x$, $y$ and $z$), and the set $\mathcal{O}$ of the remaining $q^3$ $\mathbb{F}_{q^2}$-rational points of $\mathcal{C}_n$.
The principal divisor of the variable $z$ is
\[
(z)=\sum_{P\in\mathcal{O}}P- q^3 P_{\infty}.
\]



\subsection{The curves $\mathcal{Y}_{n,s}$ and $\mathcal{X}_{a,b,n,s}$}

Two families of subcovers $\mathcal{Y}_{n,s}$ and $\mathcal{X}_{a,b,n,s}$ of the GGS curve $\mathcal{C}_n$ were introduced and studied by Tafazolian, Teh\'eran-Herrera and Torres in \cite{TTT}.

Let $q$ be a prime power, $n \geq 3$ be an odd integer and $s\geq1$ be a divisor of $m=\frac{q^n+1}{q+1}$.
We always assume $s\ne m$ (otherwise, $\cY_{n,s}$ is the Hermitian curve $\mathcal{H}_q$).
The curve $\mathcal{Y}_{n,s}$ is defined over $\FF_{q^{2n}}$ by the affine equations
\begin{equation} \label{Yns}
\mathcal{Y}_{n,s}: \left\{ \begin{array}{rcl}
Z^{m/s}=Y^{q^2}-Y\\
Y^{q+1}=X^q+X
\end{array}\right..
\end{equation}
It is an $\FF_{q^{2n}}$-maximal curve and it has genus $$g(\mathcal{Y}_{n,s})=\dfrac{q^{n+2}-q^n - sq^3 + q^2+s-1}{2s}.$$
The curve $\mathcal{Y}_{n,s}$ is clearly a subcover of the GGS curve, but 
it is also a generalization of the GGS, 
as it provides a larger family of maximal curves in which $\mathcal{C}_n$ lives from $\mathcal{Y}_{n,1}=\mathcal{C}_n$.

Let $x,y,z$ be the coordinate functions of $\mathcal{Y}_{n,s}$, $P_{\infty}$ be the unique common pole of $x,y,z$, and $P_{(\alpha,\beta,\gamma)}$ denote the $\mathbb{F}_{q^{2n}}$-rational point of $\mathcal{Y}_{n,s}$ which is a zero of $x-\alpha,y-\beta,z-\gamma$.
Then, for any $\alpha,\beta \in \mathbb{F}_{q^{2n}}$, we have the following principal divisors:
\begin{equation} \label{div x Yns}
(x - \alpha) = (q+1)m/s P_{(\alpha,0,0)} - (q+1)m/s P_{\infty} \mbox{ ;}    \end{equation}
\begin{equation} \label{div y Yns}
\displaystyle (y - \beta) = \sum_{i=1}^{q} m/s P_{(\alpha_i,\beta,0)} - qm/s P_{\infty}, \mbox{ with } \alpha_{i}^{q} + \alpha_i = \beta^{q+1}\mbox{ ;}  
\end{equation}
\begin{equation} \label{div z Yns}
\displaystyle (z) = \sum_{j=1}^{q^2}\sum_{i=1}^{q} P_{(\alpha_i,\beta_j,0)} - q^3 P_{\infty}, \mbox{ with } \alpha_i,\beta_{j} \in \mathbb{F}_{q^2} \mbox{ and } \beta_{j}^{q+1} = \alpha_i^{q} + \alpha_i \mbox{ for all } i,j.    
\end{equation}
We will denote with $O$ the set $O:=\{P_{(\alpha_i,\beta_j,0)} \,\colon\, \alpha_i,\beta_j \in \mathbb{F}_{q^2} , \; \beta_j^{q+1}=\alpha_i^q+\alpha_i\}$ of cardinality $q^3$ given by the totally ramified points in $\mathbb{F}_{q^2}(x,y,z)/\mathbb{F}_{q^2}(x,y)$ other than $P_\infty$.
From \cite[Proposition 5.1]{TTT} we have that $H(P_{\infty}) = \langle qm/s, q^3, (q+1)m/s \rangle$, and $H(P_{\infty})$ is a telescopic semigroup. 

The following lemma will be used to determine the full automorphism group of $\mathcal{Y}_{n,s}$.

\begin{lemma} \label{exact}
Denote with $dz$ the differential of the function $z$ in the function field of $\mathcal{Y}_{n,s}$. Then $(dz)$ is equal to
$$K= (2g(\cY_{n,s})-2)P_{\infty}.$$
In particular, the canonical differential $K$ is exact.
\end{lemma}

\begin{proof}
Denote with $F$ the function field of $\mathcal{Y}_{n,s}$ over the algebraic closure $\mathbb{K}$ of $\mathbb{F}_{q^{2n}}$. By Equation \eqref{div z Yns}, the function field extension $F/\mathbb{K}(z)$ is of degree $q^3$. More precisely, $F/\mathbb{K}(z)$ is a Galois extension whose Galois group $G\subseteq\Aut(\mathcal{Y}_{n,s})$ is $G=\{\theta_{b,c}\;\colon\;b,c\in\mathbb{F}_{q^2},c^q+c=b^{q+1} \}$, where
$$\theta_{b,c}(x)=x+b^qy +c, \quad \theta_{b,c}(y)=y+b, \quad \theta_{b,c}(z)=z.$$
Since $G$ fixes the function $z$, it fixes its divisor. This implies that $G$ acts on the support of both zero and pole divisors of $z$ given in \eqref{div z Yns}. In particular, $G$ fixes $P_\infty$. Since $\mathcal{Y}_{n,s}$ is $\mathbb{F}_{q^{2n}}$-maximal it has $p$-rank zero. From \cite[Lemma 11.129]{HKT} $P_\infty$ is the only ramified point in  $F/\mathbb{K}(z)$, as elements of order a power of the characteristic $p$ can only fix $P_\infty$ and no other places in $F$.
From \cite[Theorem 3.4.6]{stichtenoth},
$$({\rm Cotr} F/\mathbb{K}(z)(dz)) = (dz)_F = {\rm Con} F/\mathbb{K}(z)((dz)) + {\rm Diff}(F|\mathbb{K}(z)).$$

Since the support of both ${\rm Con} F|\mathbb{K}(z)((dz))$ and ${\rm Diff}(F|\mathbb{K}(z))$ is just $P_\infty$, while the degree of the divisor $(dz)_F$ is $2g(\mathcal{Y}_{n,s})-2$, we get
that $(dz)=K$.
\end{proof}

The maximality of $\mathcal{Y}_{n,s}$, the fact that each point of $O \cup \{P_\infty\}$ is totally ramified in the covering $\mathcal{Y}_{n,s} \rightarrow \mathcal{H}_q$ with $P \mapsto \bar P$ and the fact that $O \cup \{P_\infty\}$ is exactly where the ramification occurs, yield the existence of special functions on the function fields of both $\mathcal{Y}_{n,s}$ and the Hermitian curve $\mathcal{H}_q$, that can be arbitrarily difficult to construct by hands. This is a consequence of the so-called \textit{Fundamental equation}, see \cite[Section 9.8]{HKT}. These functions are summarized in the following lemma.

\begin{lemma}\label{fundeq}
Let $P =P_{(\alpha,\beta,\gamma,1)}$ be an $\mathbb{F}_{q^{2n}}$-rational point of $\mathcal{Y}_{n,s}$. Then there exists a function $f_P$ in the function field of  $\mathcal{Y}_{n,s}$ such that 
$$(\pi_P)=(q^n+1)P-(q^n+1)P_\infty,$$
that is the order of the equivalence class $[P-P_\infty]$ in the Picard group of $\mathcal{Y}_{n,s}$ divides $q^n+1$.
If $\gamma \ne 0$, that is $P \not\in O$, then there exists a function $\xi_{\bar P}$ on the Hermitian curve $\mathcal{H}_q$ such that 
$$(\xi_{\bar P})_{\mathcal{H}_q}=q\bar P+\Phi(\bar P)-(q+1)\bar{P}_\infty,$$
where $\Phi(\bar P)$ denotes the $\mathbb{F}_{q^2}$-Frobenius image of $\bar P$. In particular, seeing the function $\xi_{\bar P}$ on $\mathcal{Y}_{n,s}$ gives
$$(\xi_{\bar P})=qP+E_P-(q+1)m/s P_\infty,$$
where $E_P$ is an effective divisor whose support does not contain $P$ nor $P_\infty$.
\end{lemma}

The curve $\mathcal{X}_{a,b,n,s}$ is defined for any odd integer $n\geq3$ and prime power $q$ where $q=p^a$ with $p$ prime, $b\geq1$ is a divisor of $a$, $\bar{q}:=p^b$, and $s\geq1$ is a divisor of $m=\frac{q^n+1}{q+1}$. 
Choose $c \in \mathbb{F}_{q^2}$ such that $c^{q-1}=-1$. The $\FF_{q^{2n}}$-rational curve $\mathcal{X}_{a,b,n,s}$ is given by the affine equations
\begin{equation} \label{Xabns}
\mathcal{X}_{a,b,n,s}:\begin{cases}
    Z^{m/s}=Y^{q^2}-Y\\
    cY^{q+1}=\mathrm{Tr}_{q/\bar{q}}(X)
\end{cases},
\end{equation}
where $\mathrm{Tr}_{q/\bar{q}}(X)=X+X^{\bar{q}}+\cdots+X^{q/\bar{q}}$ is the trace map of the extension $\mathbb{F}_q/\mathbb{F}_{\bar{q}}$.

The curve $\mathcal{X}_{a,b,n,s}$ is $\FF_{q^{2n}}$-maximal and it has genus $$g(\mathcal{X}_{a,b,n,s})=\dfrac{q^{n+2}-\bar{q}q^n - sq^3 + q^2+(s-1)\bar{q}}{2s\bar{q}}.$$ Furthermore, $\mathcal{X}_{a,b,n,s}$ is a subcover of the GGS curve $\mathcal{C}_n$. 

Let $x,y,z$ be the coordinate functions of $\mathcal{Y}_{n,s}$, $P_{\infty}$ be the unique common pole of $x,y,z$, and $P_{(\alpha,\beta,\gamma)}$ denote the $\mathbb{F}_{q^{2n}}$-rational point of $\mathcal{Y}_{n,s}$ which is a zero of $x-\alpha,y-\beta,z-\gamma$.
Then, for any $\alpha,\beta \in \mathbb{F}_{q^{2n}}$, we have the following principal divisors:
\begin{equation} \label{div x}
(x - \alpha) = (q+1)m/sP_{(\alpha,0,0)} - (q+1)m/sP_{\infty} \mbox{ };  
\end{equation}
\begin{equation} \label{div y}
\displaystyle (y - \beta) = \sum_{i=1}^{q/\bar{q}} m/sP_{(\alpha_i,\beta,0)} - \dfrac{q}{\bar{q}}m/sP_{\infty}, \mbox{ with } \mathrm{Tr}_{q/\bar{q}}(\alpha_{i}) = \beta^{q+1}\mbox{ } ;    
\end{equation}
\begin{equation} \label{div z}
\displaystyle (z) = \sum_{j=1}^{q^2}\sum_{i=1}^{q/\bar{q}} P_{(\alpha_i,\beta_j,0)} - \dfrac{q^3}{\bar{q}} P_{\infty}, \mbox{ with } \beta_{j} \in \mathbb{F}_{q^2} \mbox{ and } c \beta_{j}^{q+1} = \mathrm{Tr}_{q/\bar{q}}(\alpha_i), \mbox{ for all } i,j.   
\end{equation}
From \cite[Proposition 5.1]{TTT} we have that $H(P_{\infty}) = \langle \frac{q}{\bar q}m/s, \frac{q^3}{\bar q}, (q+1)m/s \rangle$, which is a telescopic semigroup.

\section{The automorphism group of $\mathcal{Y}_{n,s}$ and $\mathcal{X}_{a,b,n,s}$}\label{sec:results}

In this section the automorphism groups of of $\mathcal{Y}_{n,s}$ and $\mathcal{X}_{a,b,n,s}$ are computed. 

\subsection{The automorphism group of $\mathcal{Y}_{n,s}$}

We aim to prove the following theorem.

\begin{theorem} \label{mainY}
    If $3\nmid n$ or $\frac{m}{s}\nmid(q^2-q+1)$, then the full automorphism group ${\rm Aut}(\mathcal{Y}_{n,s})$ of $\mathcal{Y}_{n,s}$ has order $q^3(q^2-1)m/s$ and is isomorphic to $S_{q^3}\rtimes C_{(q^2-1)\frac{m}{s}}$.
    
    If $3\mid n$ and $\frac{m}{s}\mid (q^2-q+1)$, then ${\rm Aut}(\mathcal{Y}_{n,s})$ has order $(q^3+1)q^3(q^2-1)m/s$ and is isomorphic to ${\rm PGU}(3,q)\rtimes C_{m/s}$. 
\end{theorem}
The case $s=1$, that is the GGS curve $\mathcal{C}_n$, has been analyzed independently in \cite{GMP} and \cite{GOS}.
Recall that if $(n,s)=(3,1)$ then the curve $\mathcal{Y}_{n,s}$ is the so-called GK-curve, whose full automorphism group is well-known, see \cite{GK}.
From this point of view Theorem \ref{mainY} can be seen as a new characterization of the GK curve, in terms of its automorphism group, in the family of maximal curves $\mathcal{Y}_{n,s}$ constructed in \cite{TTT}.

We start by observing that an automorphism group $G$ of order $q^3(q^2-1)m/s$ and isomorphic to $S_{q^3}\rtimes C_{(q^2-1)\frac{m}{s}}$ can be found by hands, independently from the condition $3 \nmid n$ or $\frac{m}{s}\nmid(q^2-q+1)$. It is readily seen indeed that the following are automorphism groups of $\mathcal{Y}_{n,s}$

$$ S_{q^3}=\left\{(x,y,z)\mapsto(x+b^q y+c,y+b,z)\;\colon\; b,c\in\mathbb{F}_{q^2},c^q+c=b^{q+1}\right\}\cong E_q\,.\,E_{q^2}$$
where
$$E_q=\left\{ (x,y,z)\mapsto(x+c,y,z) \;\colon\; c\in\mathbb{F}_{q^2},c^q+c=0 \right\}$$
is the center of $S_{q^3}$. Also, if $a$ is a primitive element of $\mathbb{F}_{q^2}$ and $\lambda_a\in\mathbb{F}_{q^{2n}}$ satisfies $\lambda_a^\frac{m}{s}=a$, another automorphism group of $\mathcal{Y}_{n,s}$ is given by
$$C:= \langle \tau: (x,y,z) \mapsto (a^{q+1}x,ay,\lambda_a z) \rangle \cong C_{(q^2-1)\frac{m}{s}}.$$

We remark that $C$ contains the cyclic group $C_{m/s}={\rm Gal}(\mathbb{F}_{q^{2n}}(x,y,z)/ \mathbb{F}_{q^{2n}}(y,z))$.

Since $S_{q^3}$ and $C$ are of co-prime order and $C$ normalizes $S_{q^3}$, we get that $G:=\langle S_{q^3},C \rangle$ has order $q^3(q^2-1)m/s$ and is equal to $S_{q^3}\rtimes C_{(q^2-1)\frac{m}{s}}$. 

Note that the group $G$ has exactly one fixed point, namely $P_\infty$. This follows from the fact that $S_{q^3}$ has $\{P_\infty \}$ as its unique short orbit from Lemma \ref{prankpel}, and $S_{q^3}$ is normal in $G$. Theorem \ref{mainY} is proven by first showing that the stabilizer of $P_\infty$ in ${\rm Aut}(\mathcal{Y}_{n,s})$ is exactly $G$. To this aim, some preliminary technical lemmas are needed.

\begin{lemma} \label{contr1}
Let $\alpha \in E_q \setminus \{id\}$. Then $i(\alpha)=(q^n+1)/s+1$.
 \end{lemma}

\begin{proof}
Note that $\alpha$ is of prime order $p$ since $E_q$ is elementary abelian, and $\alpha$ acts non-trivially on $x$, while both $y$ and $z$ are fixed by $\alpha$. This means that the fixed field of $E_q$ contains $\mathbb{F}_{q^{2n}}(y,z)$, where $z^{m/S}=y^{q^2}-y$.
Actually the fixed field of $E_q$ coincides with $\mathbb{F}_{q^{2n}}(y,z)$, as the extension $\mathbb{F}_{q^{2n}}(x,y,z)/ \mathbb{F}_{q^{2n}}(y,z)$ has degree $q=|E_q|$. 

All the elements of $E_q \setminus \{id\}$ are conjugate in $G$, because the subgroup $\langle \tau^{m(q+1)/s} \rangle \subset C$ of order $q-1$ acts transitively on $E_q \setminus \{id\}$ by conjugation. Hence each $\alpha \in E_q \setminus \{id\}$ gives the same contribution $A:=i(\alpha)$ to the different divisor of the extension $\mathbb{F}_{q^{2n}}(x,y,z)/ \mathbb{F}_{q^{2n}}(y,z)$. Since $g(\mathbb{F}_{q^{2n}}(y,z))=(m/s-1)(q^2-1)/2$, we get from the Hurwitz genus formula that
$$2g(\mathcal{Y}_{n,s})-2=\frac{(q^2-1)(q^n+1)}{s}-(q^3+1)=|E_q|\bigg( 2g(\mathbb{F}_{q^{2n}}(y,z))-2\bigg)+\sum_{\alpha \in E_q \setminus \{0\}}i(\alpha)$$
$$=q\bigg( (q^2-1)(m/s+1)-2\bigg)+A(q-1),$$
from which one gets the claim $i(\alpha)=(q^n+1)/s+1$.
\end{proof}

\begin{lemma} \label{contr2}
Let $\beta \in S_{q^3} \setminus E_q$. Then $i(\beta)=m/s+1$.
 \end{lemma}

\begin{proof}
We first observe that the fixed field $F_{q^3}$ of $S_{q^3}$ is the rational function field $\mathbb{F}_{q^{2n}}(z)$. In fact, $z$ is fixed by $S_{q^3}$, and the extension $\mathbb{F}_{q^{2n}}(x,y,z)/\mathbb{F}_{q^{2n}}(z)$ has degree $\deg((z)_{\infty})=\deg(q^3 P_{\infty})=|S_{q^3}|$. Since $E_q$ is the only proper normal subgroup in $S_{q^3}$, the ramification groups $S_{q^3}^{(i)}$ either coincide with $S_{q^3}$, or with $E_q$, or are trivial. This implies that the degree of the different divisor of the extension $\mathbb{F}_{q^{2n}}(x,y,z)/ F_{q^3}$ can be written as $(n-j)(q-1)+j(q^3-1)=n(q-1)+j(q^3-q)$ where $n$ is the number of non-trivial ramification groups (including the 0th ramification group), and $j$ the number of ramification groups coinciding with $S_{q^3}$. From Lemma \ref{contr1} each element of $E_q \setminus \{id\}$ is contained in exactly $(q^n+1)/s+1$ ramification groups, which implies that $n=(q^n+1)/s+1$ (because all the higher ramification groups contained properly in $E_q$ need to be trivial). From the Hurwitz genus formula and Lemma \ref{contr1} we obtain 
$$2g(\mathcal{Y}_{n,s})-2=|S_{q^3}|\bigg( 2g(\mathbb{F}_{q^{2n}}(z))-2\bigg)+\left(\frac{q^n+1}{s}+1\right)(q-1)+ j(q^3-q),$$
which yields $j=m/s+1$ by direct computation. Since we have only 2 possible higher ramification groups, the elements of $S_{q^3} \setminus E_q$ give all the same contribution to the different divisor, namely $m/s+1$.
\end{proof}

Recall that from Lemma \ref{highnormal}, $\aut(\mathcal{Y}_{n,s})_{P_\infty}=\tilde{S} \rtimes \tilde{C}$, where $\tilde S$ is the Sylow $p$-subgroup of $\aut(\mathcal{Y}_{n,s})_{P_\infty}$ and $\tilde C$ is a cyclic $p^{\prime}$-group. Our first aim is to show that $\tilde{C}$ coincides with $C$.

We denote by $\bar{P}_{\infty}$ the point of $\mathcal{H}_q$ lying below $P_\infty$.

\begin{lemma} \label{compfix}
Let $G \subseteq \aut(\mathcal{Y}_{n,s})_{P_\infty}=\tilde{S} \rtimes \tilde{C}$, where $\tilde S$ is the Sylow $p$-subgroup of $\aut(\mathcal{Y}_{n,s})_{P_\infty}$ and $\tilde C$ is cyclic of order $d$ where $(d,p)=1$. Then $|\tilde C|=|C|=(q^2-1)\frac{m}{s}$.
\end{lemma}

\begin{proof} 
Up to conjugation we can choose $\tilde C$ such that $C \subseteq \tilde C$. Since $\tilde C$ is cyclic, $C_{m/s} \subseteq C$ is a normal subgroup of $\tilde C$. 
The quotient group $\tilde C/C_{m/s}$ is an automorphism group of $\mathcal{Y}_{n,s}/C_{m/s}=\mathcal{H}_q$ fixing $\bar{P}_\infty$. 
Since the stabilizer of $\bar{P}_\infty$ in ${\rm Aut}(\cH_q)$ has order $q^3(q^2-1)$ and $\tilde{C}/C_{m/s}$ is a $p^{\prime}$-group, we get that $|\tilde C/C_{m/s}| \leq q^2-1$. However since $\tilde C$ contains $C$ we have $q^2-1=|C/C_{m/s}| \leq |\tilde C/C_{m/s}|$, which yields $C=\tilde C$.
\end{proof}

To complete the proof of our intermediate statement $\aut(\mathcal{Y}_{n,s})_{P_\infty}=G$, we need to show that $\tilde S=S_{q^3}$.

\begin{lemma} \label{normpsyl}
Let $\gamma \in \tilde S \setminus S_{q^3}$. Then $i(\gamma)=2$. In particular both $E_q$ and $S_{q^3}$ are normal subgroups of $\tilde S$.
\end{lemma}

\begin{proof}
Since the fixed field of $S_{q^3}$ is rational and $S_{q^3}\subseteq\tilde{S}$, the fixed field of $\tilde S$ is also rational. From Lemmas \ref{contr1} and \ref{contr2} we know that $i(\alpha) = (q^n+1)/s+1$ and $i(\beta) = m/s+1$ for all $\alpha \in E_q \setminus \{id\}$ and $\beta \in S_{q^3} \setminus E_q$. Furthermore for all $\gamma \in \tilde S \setminus S_{q^3}$ one has $i(\gamma) \geq 2$ as $\tilde S=\tilde S_{P_\infty}^{(0)}=\tilde S_{P_\infty}^{(1)}$.
Then the Hurwitz genus formula applied to $\cY_{n,s}\to\cY_{n,s}/\tilde{S}$ gives
$$\frac{(q^n+1)(q^2-1)}{s}-(q^3+1) \geq -2|\tilde S|+\bigg(\frac{(q^n+1)}{s}+1\bigg)(q-1)+\bigg(\frac{m}{s}+1\bigg)(q^3-q)+2(|\tilde S|-q^3).$$
Since the right and left hand-sides coincide, we deduce that equality must hold, 
that is $i(\gamma)=2$ for all $\gamma\in\tilde{S}\setminus S_{q^3}$.
In particular $\tilde S=\tilde S_{P_\infty}^{(0)}=\tilde S_{P_\infty}^{(1)}$, $S_{q^3}=\tilde S_{P_\infty}^{(2)}=\ldots=\tilde S_{P_\infty}^{(m/s)}$, $E_q=\tilde S_{P_\infty}^{(m/s+1)}=\ldots=\tilde S_{P_\infty}^{(q^n+1)/s}$ and $\tilde S_{P_\infty}^{(i)}=\{id\}$ for all $i \geq (q^n+1)/s+1$. Now, $S_{q^3}$ is normal in $\tilde S$ by Lemma \ref{highnormal} item 2. The subgroup $E_q$ is hence also normal in $\tilde S$, being the center of $S_{q^3}$ and hence a characteristic subgroup of $S_{q^3}$. 
\end{proof}

\begin{remark} \label{remop}
The statement about the normality of $E_q$ and $S_{q^3}$ can be strengthened by looking at the entire stabilizer $\aut(\mathcal{Y}_{n,s})_{P_\infty}$ and not only at $\tilde S$. Since Lemma \ref{highnormal} implies that higher ramification groups are normal in the entire stabilizer of a point and $\tilde S=\aut(\mathcal{Y}_{n,s})_{P_\infty}^{(1)}$, we have that $S_{q^3}$ and $E_q$ are normal in $\aut(\mathcal{Y}_{n,s})_{P_\infty}$.
\end{remark}

With the previous lemmas and remarks, we are in a position to prove our aimed intermediate statement $\aut(\mathcal{Y}_{n,s})_{P_\infty}=G$, which we observe being true independently from the condition $3\nmid n$ or $\frac{m}{s}\nmid(q^2-q+1)$ (that indeed we have never used so far).

\begin{proposition} \label{stabY}
The stabilizer $\aut(\cY_{n,s})_{P_{\infty}}$ of $P_\infty$ in $\aut(\cY_{n,s})$ is $G$.
\end{proposition}

\begin{proof}
By Lemma \ref{normpsyl} and Remark \ref{remop}, $E_q$ is a normal subgroup of $\aut(\mathcal{Y}_{n,s})_{P_\infty}$. Hence $\aut(\mathcal{Y}_{n,s})_{P_\infty}/E_q$ is an automorphism group of the fixed field of $E_q$, that is, $\mathbb{F}_{q^{2n}}(y,z)$ with $y^{q^2}-y=z^{m/s}$.
Since $\gcd(q^2+1,q^n+1)=2$ for $n$ odd, we get from \cite[Theorem 3.2]{BMZsep} that 

$$\frac{m}{s}q^2(q^2-1)=|G/E_q| \leq |\aut(\mathcal{Y}_{n,s})_{P_\infty}/E_q| \leq |\aut(\mathbb{F}_{q^{2n}}(y,z))|=\frac{m}{s}q^2(q^2-1),$$

which implies $\aut(\mathcal{Y}_{n,s})_{P_\infty}=G$.
\end{proof}


Define the following set of rational points of $\mathcal{Y}_{n,s}$:
$$O=\{P_{(\alpha_i,\beta_j,0)}\in\mathcal{Y}_{n,s} \;\colon\; \alpha_i,\beta_j \in \mathbb{F}_{q^2} , \, \beta_j^{q+1}=\alpha_i^q+\alpha_i\}.$$
The following is an easy but useful lemma.

\begin{lemma} \label{normaction}
The group $C_{m/s}$ is normal in $\aut(\mathcal{Y}_{n,s})$ if and only if $\aut(\mathcal{Y}_{n,s})$ acts on the set $O \cup \{P_\infty\}$. 
Furthermore, if $C_{m/s}$ is normal in $\aut(\mathcal{Y}_{n,s})$, then $\aut(\mathcal{Y}_{n,s})/C_{m/s}$ is an automorphism group of the Hermitian function field $\mathbb{F}_{q^{2n}}(x,y)$ and either $\aut(\mathcal{Y}_{n,s})=G$ or $|\aut(\mathcal{Y}_{n,s})/C_{m/s}|=q^3(q^2-1)(q^3+1)$.
In the latter case, $\aut(\mathcal{Y}_{n,s})/C_{m/s}$ is isomorphic to ${\rm PGU}(3,q)$ and acts on $\mathbb{F}_{q^{2n}}(x,y)$ as ${\rm PGU}(3,q)$ in its natural action.
\end{lemma}

\begin{proof}
Recall that $O \cup \{P_\infty\}$ consists exactly of the fixed points of $C_{m/s}$ on $\mathcal{Y}_{n,s}$, and $C_{m/s}$ has no other short orbits. 
This immediately implies that if $C_{m/s}$ is normal in $\aut(\mathcal{Y}_{n,s})$ then $\aut(\mathcal{Y}_{n,s})$ acts on $O \cup \{P_\infty\}$: in fact, if $\alpha \in \aut(\mathcal{Y}_{n,s})$, $\beta \in C_{m/s}$ and $P \in O \cup \{P_\infty\}$, then $\alpha(P)=\alpha(\beta(P))=\beta^{\prime}(\alpha(P))$ for some $\beta^{\prime}\in C_{m/s}$ and hence $\alpha(P)$ is fixed by $\beta^{\prime}$, i.e. $\alpha(P)\in O\cup\{P_{\infty}\}$.

Conversely, suppose that $\aut(\mathcal{Y}_{n,s})$ acts on the set $O \cup \{P_\infty\}$ and consider the subgroup
$$T=\{\sigma \in \aut(\mathcal{Y}_{n,s}) \mid \sigma(P)=P, \ \textrm{for \ all} \ P \in O \cup \{P_\infty\}\}$$
of ${\rm Aut}(\cY_{n,s})$.
Let $g \in \aut(\mathcal{Y}_{n,s})$ and $\sigma \in T$. For all $P \in O \cup \{P_\infty\}$ it holds that $g(P) \in O \cup \{P_\infty\}$ and hence $\sigma(g(P))=g(P)$, which implies $g^{-1}\sigma g(P)=P$ for all $P \in O \cup \{P_\infty\}$, that is, $g \sigma g^{-1} \in T$.
Thus, $T$ is a normal subgroup of ${\aut}(\cY_{n,s})$.
Moreover, the characteristic $p$ does not divide the order of $T$, because $|S|>1$ and the curve has $p$-rank zero, see \cite[[Lemma 11.129]{HKT}. Then, by Lemma \ref{highnormal}, $T$ is cyclic. Therefore $C_{m/s} \subseteq T$ is a characteristic subgroup of $T$ and hence a normal subgroup of $\Aut(\mathcal{Y}_{n,s})$. 

To prove the second part of the statement recall that the fixed field of $C_{m/s}$ is the Hermitian function field $\mathbb{F}_{q^{2n}}(x,y)$ and that $\aut(\mathbb{F}_{q^{2n}}(x,y))= {\rm PGU}(3,q)$, see \cite[Proposition 11.30]{HKT}. Thus, $\aut(\mathcal{Y}_{n,s})/C_{m/s}$ is a subgroup of ${\rm PGU}(3,q)$ containing $G/C_{m/s}$, which is a group of order $q^3(q^2-1)$ fixing the point $\bar{P}_\infty$ below $P_\infty$.
Since ${\rm PGU}(3,q)_{\bar{P}_\infty}$ is a maximal subgroup of ${\rm PGU}(3,q)$ of order $q^3(q^2-1)$ (see \cite[Theorem A.10]{HKT}), we get either $\aut(\mathcal{Y}_{n,s})/C_{m/s}=G/C_{m/s}$ (in this case, $\aut(\mathcal{Y}_{n,s})$ fixes $P_\infty$) or $\aut(\mathcal{Y}_{n,s})/C_{m/s}={\rm PGU}(3,q)$.
\end{proof}

Lemma \ref{normaction} is a key ingredient, because it defines the strategy we are going to use to complete the proof of the theorem. First we prove that if $C_{m/s}$ is normal, then $\aut(\mathcal{Y}_{n,s})/C_{m/s} \cong {\rm PGU}(3,q)$ if and only if $3 \mid n$ and $\frac{m}{s}\mid(q^2-q+1)$. Then we prove, independently from $n$, that $\aut(\mathcal{Y}_{n,s})$ acts on $O \cup \{P_\infty\}$. In this way, whenever $3 \nmid n$ or $\frac{m}{s}\nmid(q^2-q+1)$, the claim $\aut(\mathcal{Y}_{n,s})=\aut(\mathcal{Y}_{n,s})_{P_{\infty}}$ will follow immediately from Lemma \ref{normaction}.

\begin{lemma} \label{3divnlift}
Assume that $C_{m/s}$ is normal in $\aut(\mathcal{Y}_{n,s})$. 
Then $\aut(\mathcal{Y}_{n,s})/C_{m/s} \cong {\rm PGU}(3,q)$ if and only if $3 \mid n$ and $\frac{m}{s}\mid (q^2-q+1)$. 
\end{lemma}

\begin{proof}
The automorphism group ${\rm PGU}(3,q)$ of the function field $\mathbb{F}_{q^{2n}}(x,y)$ is $2$-transitive on the rational places of $\mathbb{F}_{q^{2n}}(x,y)$, and hence can be generated as ${\rm PGU}(3,q)=\langle {\rm PGU}(3,q)_{Q_\infty},\tau \rangle$, where $\tau(x)=1/x$ and $\tau(y)=y/x$; see \cite[Page 664, item 9]{HKT}. Since $C_{m/s}$ is normal in $\aut(\mathcal{Y}_{n,s})$, we  apply Lemma \ref{normaction}.
From $G/C_{m/s}={\rm PGU}(3,q)_{\bar{P}_\infty}$ we get that $\aut(\mathcal{Y}_{n,s})/C_{m/s} \cong {\rm PGU}(3,q)$ if and only if $\tau$ can be lifted to an automorphism $\tilde\tau$ of $\mathbb{F}_{q^{2n}}(x,y,z)$ having the same action of $\tau$ on the totally ramified points in $O \cup \{P_\infty\}$. 

Suppose that $\aut(\mathcal{Y}_{n,s})/C_{m/s} \cong {\rm PGU}(3,q)$, i.e. such $\tilde\tau$ exists. Then $\tilde\tau(P_\infty)=P_{(0,0,0)}$ and $O \setminus \{P_{(0,0,0)}\}$ is fixed setwise by $\tilde\tau$. From Equations \eqref{div x Yns} and \eqref{div y Yns}, this implies that the divisors of $\tilde\tau(x)$ and $1/x$ coincide, and the same holds for those of $\tilde\tau(y)$ and $y/x$. Therefore $\tilde\tau(x)=\lambda/x$ and $\tilde\tau(y)=\mu y/x$ for some non-zero constants $\lambda,\mu\in\mathbb{F}_{q^{2n}}$. Since the equality $y^{q+1}=x^q+x$ is preserved by $\tilde\tau$, we get that $\lambda=\mu^{q+1}$. Also from Equation \eqref{div z Yns} we have that
$$(\tilde\tau(z))=\left(\sum_{P\in O, \ P \ne P_{(0,0,0)}} P\right) + P_\infty - q^3P_{(0,0,0)}$$
and 
$$(z/\tilde\tau(z))=(q^3+1)(P_{(0,0,0)}-P_\infty).$$
By Equation \eqref{div x}, this implies that the order $o_{P_{(0,0,0)}}$ of $[P_{(0,0,0)}-P_\infty]$ in the Picard group of $\mathcal{Y}_{n,s}$ divides $\gcd((q+1)m/s,q^3+1)=(q+1)\cdot\gcd(m/s,q^2-q+1)\leq (q+1)m/s$. Also, $o_{P_{(0,0,0)}}\in H(P_{\infty})$ and $o_{P_{(0,0,0)}}$ is coprime with $q$.
By \cite[Proposition 5.1]{TTT}, the semigroup $H(P_\infty)$ satisfies $H(P_\infty)=\langle qm/s,(q+1)m/s,q^3 \rangle$; thus, the smallest element of $H(P_\infty)$ coprime with $q$ is $(q+1)m/s$.
Therefore, $o_{P_{(0,0,0)}}=(q+1)m/s=\gcd((q+1)m/s,q^3+1)$, that is, $m/s$ divides $q^2-q+1$.
By Lemma \ref{fundeq}, also $\gcd(q^n+1,q^3+1)$ is an element of $H(P_\infty)$ which is coprime with $q$, and hence not smaller than $(q+1)m/s$. Since $m/s>1$, this implies $\gcd(q^n+1,q^3+1)>q+1$ and hence $3\mid n$.




Conversely, suppose that $m/s$ divides $q^2-q+1$ and $3$ divides $n$. 
Then the $\mathbb{F}_{q^{2n}}$-maximal curve $\mathcal{Y}_{n,s}$ is a quotient $\mathcal{GK}/C_{\frac{q^2-q+1}{m/s}}$ of the $\mathbb{F}_{q^6}$-maximal curve $\mathcal{GK}$; thus, $\mathcal{Y}_{n,s}$ is also $\mathbb{F}_{q^6}$-maximal.
The fundamental equation \cite[Page xix Item (ii)]{HKT} implies that there exists a function $\rho_0$ such that $(\rho_0)=(q^3+1)(P_\infty-P_{(0,0,0)})$.
Our aim is to show that $\tau$ can be lifted to an automorphism $\tilde\tau$ of $\mathcal{Y}_{n,s}$ by defining
$$\tilde\tau:(x,y,z) \mapsto \bigg( \frac{1}{x}, \frac{y}{x}, \xi \cdot \rho_0 \cdot z \bigg),$$
for some suitable constant $\xi$. 

To this aim, note first that the equation $y^{q+1}=x^q+x$ is trivially preserved by $\tilde\tau$, as $\tau$ is an automorphism of $\mathbb{F}_{q^{2n}}(x,y)$ and $\tilde{\tau}$ acts as $\tau$ on $x$ and $y$. Moreover, choosing $\xi$ carefully, we can force also the other equation of $\mathcal{Y}_{n,s}$ to be preserved by $\tilde\tau$. Indeed, we have
$$((\rho_0 z)^{m/s})=\frac{m}{s}(\rho_0 z)=\frac{m}{s}(z)+\frac{m}{s}(q^3+1)(P_\infty-P_{(0,0,0)})$$
$$=\frac{m}{s}\sum_{P\in O}P - \frac{m}{s}q^3 P_{\infty}+\frac{m}{s}(q^3+1)P_\infty-\frac{m}{s}(q^3+1)P_{(0,0,0)}$$
$$=\frac{m}{s}\left(\sum_{P \in O\cup\{P_{\infty}\}, \ P \ne P_{(0,0,0)}}P-q^3P_{(0,0,0)}\right)=(\tau(y^{q^2}-y))=(\tilde\tau(y^{q^2}-y)),$$
and hence there exists a nonzero constant $\eta\in\mathbb{F}_{q^{2n}}$ such that $(\rho_0 z)^{m/s}=\eta(\tilde\tau(y^{q^2}-y))$. Choosing $\xi$ such that $\xi^{m/s}=\eta^{-1}$ we get
$$\tilde\tau(z^{m/s})=\tilde\tau(z)^{m/s}=\xi^{m/s}(\rho_0 z)^{m/s}=\eta^{-1}\cdot \eta(\tilde\tau(y^{q^2}-y))=\tilde\tau(y^{q^2}-y),$$
so that also the second equation of $\mathcal{Y}_{n,s}$ is preserved by $\tilde\tau$, yielding $\tilde\tau \in \aut(\mathcal{Y}_{n,s})$.
Since $\tilde\tau$ acts as $\tau$ on $\mathbb{F}_{q^{2n}}(x,y)$ and $\tau\notin{\rm PGU}(3,q)_{\bar{P}_\infty}=G/C_{m/s}$, we get $\tau\notin G$ and hence ${\aut}(\mathcal{Y}_{n,s})\ne G$.
By Lemma \ref{normaction}, this implies $\aut(\mathcal{Y}_{n,s})/C_{m/s} \cong {\rm PGU}(3,q)$.
\end{proof}

At this point we are left with proving that $C_{m/s}$ is normal in $\aut(\mathcal{Y}_{n,s})$, or equivalently  $\aut(\mathcal{Y}_{n,s})$ acts on $O \cup \{P_\infty\}$; then Lemmas \ref{normaction}, \ref{3divnlift} will complete the proof of Theorem \ref{mainY}.

In some cases this property can be obtained almost for free.

\begin{lemma} \label{CondB}
If $q^3(q^3-1) \leq (q^n+1)(q^2-1)/s-(q^3+1)$ then  
$C_{m/s}$ is normal in $\aut(\mathcal{Y}_{n,s})$.    
\end{lemma}

\begin{proof}
We start by proving that a point $P\in\cY_{n,s}\setminus(O\cup\{P_{\infty}\})$ cannot have the same Weierstrass semigroup as $P_{\infty}$. As $q^3\in H(P_\infty)$ by \cite[Proposition 5.1]{TTT}, it is enough to show that $q^3\notin H(P)$.
We can write $P=P_{(a,b,c)}$ as $P\ne P_{\infty}$, and $c\ne0$ as $P\notin O$.
Consider the differential $w:=(z-c)^{q^3-1}dz$ on $\cY_{n,s}$.
By Equation \eqref{div z Yns} and Lemma \ref{exact}, the valuation of $w$ at $P_{\infty}$ is $v_{P_\infty}(w)=-q^3(q^3-1)+2g(\cY_{n,s})-2$, and hence $v_{P_{\infty}}(w)\geq0$ by the assumption. Then $w$ is a regular differential, with valuation $q^3-1$ at $P$. Lemma \ref{gapsdiff} implies that $q^3=(q^3-1)+1\notin H(P)$.


Since the Weierstrass semigroup of a point is invariant under automorphisms, we have proved for any $P\in\cY_{n,s}\setminus(O\cup\{P_{\infty}\})$ that $P$ is not in the same orbit as $P_{\infty}$ under $\aut(\cY_{n,s})$. 
We now prove that any $P\in\cY_{n,s}\setminus(O\cup\{P_{\infty}\})$ is not in the same orbit of any point of $O$.
Suppose on the contrary that some point of $O$ is in the orbit $O_P$ of $P\in\cY_{n,s}(O\cup\{P_{\infty}\})$ under $\aut(\cY_{n,s})$. Then $O\subseteq O_P$, because $O$ is an orbit under $G\subseteq\aut(\cY_{n,s})$. Since $P_\infty\notin O_P$, this implies that $P$ is not in the same orbit of any point of $O$ under $\aut(\cY_{n,s})$. Therefore $\{P_\infty\}$ is an orbit under $\aut(\cY_{n,s})$, i.e. $\aut(\cY_{n,s})$ fixes $P_\infty$ and $\aut(\cY_{n,s})=G$ by Proposition \ref{stabY}. Thus $C_{m/s}$ is normal in $\aut(\cY_{n,s})$, and hence $\aut(\cY_{n,s})$ acts on $O\cup\{P_\infty\}$ by Lemma \ref{normaction}, in contradiction with $O\subseteq O_P$ with $P\notin O\cup\{P_\infty\}$.

We have then proved that points of $O\cup\{P_\infty\}$ and points out of $O\cup\{P_\infty\}$ cannot be in the same orbit under $\aut(\cY_{n,s})$, that is, $\aut(\cY_{n,s})$ acts on $O\cup\{P_\infty\}$. By Lemma \ref{normaction}, this is equivalent to $C_{m/s}$ being normal in $\aut(\cY_{n,s})$.
\end{proof}

\begin{remark}
An equivalent proof for Lemma \ref{CondB} can be proposed by showing that the gap sequences at $P\notin O\cup\{P_\infty\}$ and $Q\in O$ are different, using the generalized Weierstrass semigroup $H(Q,P_\infty)$. This semigroup has been computed in \cite{MT}.   
\end{remark}

In the following we can then assume that $q^3(q^3-1) > (q^n+1)(q^2-1)/s-(q^3+1)$; to unify the cases $q=2$ and $q>2$, we will assume $m/s \leq q^3-q^2+2q-1$.

We denote by $O_\infty$ the orbit of $P_\infty$ under $\aut(\cY_{n,s})$.

\begin{lemma} \label{onenontame}
The short orbit $O_\infty$ is the only non-tame orbit of $\aut(\mathcal{Y}_{n,s})$.
\end{lemma}

\begin{proof}
We know already that $O_\infty$ is a non-tame short orbit of $\aut(\mathcal{Y}_{n,s})$, because $G$ fixes $P_\infty$ and has order divisible by $p$. 
If $S$ is a Sylow $p$-subgroup of $\aut(\cY_{n,s})$ containing the Sylow $p$-subgroup $S_{q^3}$ of $G$, then $S=S_{q^3}$, because $S$ fixes $P_{\infty}$ by Corollary \ref{1actionp2} and $G$ is the full stabilizer of $P_\infty$ in $\aut(\cY_{n,s})$.

Let $O^\prime$ be a non-tame short orbit of $\aut(\cY_{n,s)}$, and $P^\prime$. By Lemma \ref{highnormal} we can write the stabilizer of $P^\prime$ as $\aut(\cY_{n,s})_{P^\prime}=S^\prime\rtimes C^\prime$, where $S^\prime$ is the Sylow $p$-subgroup of $\aut(\cY_{n,s})_{P^\prime}$. Arguing as above, $S^\prime$ is a Sylow $p$-subgroup of $\aut(\cY_{n,s})$.
Then the Sylow $p$-subgroups $S_{q^3}$, $S^\prime$ are conjugate, say $\alpha^{-1}S^\prime\alpha=S_{q^3}$ with $\alpha\in\aut(\cY_{n,s})$.
This implies $\alpha(P_\infty)=P^\prime$, and hence their orbits coincide, i.e. $O^\prime=O_\infty$. The claim is proved.
\end{proof}

\begin{remark}
By the properties of $p$-groups and Sylow $p$-subgroups, it can be noted that Lemma \ref{onenontame} holds also for other curves: if $\mathcal{X}$ is a curve in characteristic $p$ such that $p$ divides $|{\rm Aut}(\mathcal{X})|$ and every $p$-element of ${\rm Aut}(\mathcal{X})$ has exactly one fixed point, then ${\rm Aut}(\mathcal{X})$ has exactly one non-tame orbit.
\end{remark}

\begin{remark}\label{rem:short}
Since $G$ is the full stabilizer of $P_\infty$ in $\aut(\cY_{n,s})$ by Proposition \ref{stabY}, we get from the orbit-stabilizer theorem that
$$|\aut(\mathcal{Y}_{n,s})| = |O_\infty| \cdot |G|.$$ 
If $O_\infty=\{P_\infty\}$ then $\aut(\mathcal{Y}_{n,s})=G$, and hence $\aut(\mathcal{Y}_{n,s})$ acts on $O \cup \{P_\infty\}$, which proves Theorem \ref{mainY}.
We can then assume in the rest of this section that $|O_\infty|>1$.

Since $\{P_{\infty}\}$ and $O$ are the only short orbits of $G$, we get that either $O_\infty=\{P_\infty\} \cup O$ or $O_\infty$ contains at least one long orbit of $G$.
In the first case we get immediately that $\aut(\mathcal{Y}_{n,s})$ acts on $O \cup \{P_\infty\}$, whence Theorem \ref{mainY} follows. In the latter case we have that  
$$|\aut(\mathcal{Y}_{n,s})| = |O_\infty| \cdot |G| > |G|^2 >84(g(\mathcal{Y}_{n,s})-1),$$
and hence we can apply Theorem \ref{Hurwitz} to investigate the short orbits of $\aut(\cY_{n,s})$. Since there are automorphisms fixing points in $O \cup \{P_\infty\}$, we know that $O\cup\{P_\infty\}$ is contained in the union of short orbits of $\aut(\cY_{n,s})$.
Since $G$ fixes $P_\infty$ and has order divisible by $p$, we also know that $P_\infty$ is contained in a non-tame short orbit of $\aut(\mathcal{Y}_{n,s})$.
\end{remark}

By Remark \ref{rem:short}, the cases $O_\infty=\{P_\infty\}$ and $O_\infty=\{P_\infty\}\cup O$ have already been worked out. Therefore we can assume from now on that $O_\infty$ contains $k \geq 1$ long orbits of $G$.

This implies $|\aut(\cY_{n,s})|>84(g-1)$, so that Theorem \ref{Hurwitz} applies for the short orbits of $\aut(\mathcal{Y}_{n,s})$. Note that if $q \geq 7$, then already $|G| >84(g-1)$, and Theorem \ref{Hurwitz} applies for $\aut(\cY_{n,s})$ without any assumption on $O_\infty$.

\begin{lemma} \label{numborbits}
$\aut(\mathcal{Y}_{n,s})$ has exactly two short orbits: the non-tame orbit $O_\infty$, and a tame short orbit $O_1$.
\end{lemma}

\begin{proof}
By Theorem \ref{Hurwitz}, $\aut(\mathcal{Y}_{n,s})$ has at most three short orbits, in particular:
\begin{enumerate}
\item exactly three short orbits, two tame of length $|\aut(\mathcal{Y}_{n,s})|/2$ and one non-tame, with $p \geq 3$; or
\item exactly two short orbits, both non-tame; or 
\item only one short orbit, which is non-tame, oflength dividing $2g-2$; or
\item  exactly two short orbits, one tame and one non-tame.
\end{enumerate}
We start by observing that Case (2) cannot occur by Lemma \ref{onenontame}. If Case (3) occurs then $|O_\infty|$ must divide $2g-2$, which is impossible since $|O_\infty| \geq 1+|G|>2g-2$.
Thus, the proof is complete once we show that Case (1) cannot occur. 

Suppose that Case (1) occurs, and let $O_1,O_2$ be the tame orbits of $\aut(\cY_{n,s})$ of the same length $|\aut(\cY_{n,s})|/2$. Recall that $P_\infty$ is in the non-tame orbit $O_\infty$ of $\aut(\cY_{n,s})$, $O$ is an orbit of $G$, and $G$ acts semiregularly out of $O\cup\{P_\infty\}$. If $O\subseteq O_1$, this implies that the length of $O_1$ is congruent to $|O|$ modulo $|G|$ while the length of $O_2$ is divisible by $|G|$, a contradiction to $|O_1|=|O_2|$; therefore $O\not\subseteq O_1$ and analogously $O\not\subseteq O_2$. Thus, $O\subseteq O_\infty$.
Write $|O_\infty|=1+|O|+k|G|=q^3+1+k|G|$ for some $k \geq 1$. The Hurwitz genus formula applied to $\aut(\mathcal{Y}_{n,s})$, together with Lemmas \ref{contr1} and \ref{contr2}, gives
$$\frac{(q^n+1)(q^2-1)}{s}-(q^3+1)=$$
$$-2|\aut(\mathcal{Y}_{n,s})|+2\frac{|\aut(\mathcal{Y}_{n,s})|}{2}(2-1)+|O_\infty|\bigg(|G|-1+(q^3-1)\frac{m}{s}+(q-1)q\frac{m}{s}\bigg).$$
By the orbit-stablizer theorem $|O_\infty||G|=|\aut(\cY_{n,s})|$, whence
$$
\frac{(q^n+1)(q^2-1)}{s}-(q^3+1)=(q^3+1+k|G|)\bigg(|G|-1+(q^3-1)\frac{m}{s}+(q-1)q\frac{m}{s}\bigg).
$$
Using $|G|=q^3(q^2-1)m/s$, this yields a contradiction to $k \geq 1$.
\end{proof}

We are now in the position of proving that the assumption $k \geq 1$ yields a contradiction, which in turn gives that $\aut(\mathcal{Y}_{n,s})$ acts on $O \cup \{P_\infty\}$ as shown above.

\begin{proposition} \label{normalizzano}
$\aut(\mathcal{Y}_{n,s})$ acts on $O \cup \{P_\infty\}$.
\end{proposition}

\begin{proof}
Suppose by contradiction that $\aut(\mathcal{Y}_{n,s})$ does not act on $O \cup \{P_\infty\}$.
By Lemma \ref{numborbits}, $\aut(\cY_{n,s})$ has exactly two short orbits: the non-tame orbit $O_\infty$ of length $|O_\infty|=1+\ell|O|+k|G|$ where $k \geq 1$ and $\ell \in \{0,1\}$, and the tame orbit $O_1$ of length $|O_1|=(1-\ell)|O|+k_1|G|$ where $k_1 \geq 0$.
The Hurwtiz genus formula 
applied to $\aut(\mathcal{Y}_{n,s})$ gives
$$\frac{(q^n+1)(q^2-1)}{s}-(q^3+1)=-2|\aut(\mathcal{Y}_{n,s})|+|O_1|\bigg( \frac{|\aut(\mathcal{Y}_{n,s})|}{|O_1|}-1\bigg)$$
$$+(\ell q^3+1+k|G|)\bigg(|G|-1+(q^3-1)\frac{m}{s}+(q-1)q\frac{m}{s}\bigg).$$
 We analyze the cases $\ell=1$ and $\ell=0$ separately.
\begin{itemize}
    \item \textbf{The case $\ell=0$.} In this case $O \subseteq O_1$. Then the stabilizer of any point in $O_1$ is conjugate to the stabilizer of a point in $O$, which contains $C$ and hence has order divisible by $(q^2-1)m/s$.
    Thus, there exists some $h\geq1$ such that $|\aut(\mathcal{Y}_{n,s})_{P_1}|=h(q^2-1)m/s$ for any $P_1\in O_1$.
    By the orbit-stabilizer theorem,
    $$(1+k|G|)|G|=|\aut(\mathcal{Y}_{n,s})|=(q^3+k_1|G|)h(q^2-1)m/s,$$
    and hence
    $$1+kq^3(q^2-1)m/s=(1+k_1(q^2-1)m/s)h.$$
    In particular, $h$ is congruent to $1$ modulo $(q^2-1)m/s$. 
 
 If $h=1$, then $|\aut(\mathcal{Y}_{n,s})_{P_1}|=(q^2-1)m/s$ and $k_1=kq^3$. The Hurwitz genus formula now gives
    $$\frac{(q^n+1)(q^2-1)}{s}-(q^3+1)=-(q^3+kq^3|G|)+(1+k|G|)\bigg(-1+(q^3-1)\frac{m}{s}+(q-1)q\frac{m}{s}\bigg),$$
    and so
  $$0=-kq^3|G|+k|G|\bigg(-1+(q^3-1)\frac{m}{s}+(q-1)q\frac{m}{s}\bigg)=k|G|\bigg(-q^3-1+(q^3+q^2-q-1)\frac{m}{s} \bigg)>0,$$  
  a contradiction.

Hence $h=t(q^2-1)m/s+1$ with $t \geq 1$.
This implies $kq^3=(1+k_1(q^2-1)m/s)t+k_1$. 
The orbit-stabilizer theorem now gives
$$|\aut(\mathcal{Y}_{n,s})|=(q^3+k_1|G|)h(q^2-1)m/s > (q^3+k_1|G|)\left((q^2-1)m/s\right)^2.$$
If $k_1> 0$ then 
$$|\aut(\mathcal{Y}_{n,s})| > (q^3+|G|)\left((q^2-1)m/s\right)^2>8g^3,$$
and a contradiction is obtained from Theorem \ref{Hurwitz}.
If $k_1=0$ then $t=kq^3$ and 
\begin{equation}\label{eq:fissP1}|\aut(\mathcal{Y}_{n,s})_{P_1}|=\left(kq^3(q^2-1)m/s+1\right)(q^2-1)m/s.\end{equation}
Since the orbit of $P_1$is tame, the stabilizer $\aut(\mathcal{Y}_{n,s})_{P_1}$ is cyclic of order prime-to-$p$. Then, by \cite[Theorem 11.79]{HKT}, $\aut(\mathcal{Y}_{n,s})_{P_1}$ has order at most $4g+4=2(q^2-1)(q+1)m/s-2q^3+6$, which is a contradiction to Equation \eqref{eq:fissP1}.

    \item \textbf{The case $\ell=1$.} In this case $O\subseteq O_\infty$. From the orbit-stabilizer theorem we have
     $$(1+q^3+k|G|)|G|=|\aut(\mathcal{Y}_{n,s})|=k_1|G||\aut(\mathcal{Y}_{n,s})_{P_1}|,$$
     and hence
\begin{equation}\label{eq:uan}|\aut(\mathcal{Y}_{n,s})_{P_1}|=\left(1+q^3+k|G|\right)/k_1.\end{equation}
     On the other hand, the Hurwitz genus formula gives
$$0=-k_1|G|-k|G|+(q^3+k|G|)(q^2-1)(q+1)m/s,$$
that is
\begin{equation}\label{eq:ciu}k_1=q+1+k\left((q^2-1)(q+1)m/s-1\right).\end{equation}
Using Equations \eqref{eq:uan} and \eqref{eq:ciu} together we get
$$|\aut(\mathcal{Y}_{n,s})_{P_1}|=\frac{1+q^3+k\frac{q^3(q^2-1)m}{s}}{q+1+k\bigg(\frac{m(q^2-1)(q+1)}{s}-1\bigg)}$$
$$=q^2-q+1+\frac{k\bigg(-\frac{m(q^2-1)}{s}+q^2-q+1\bigg)}{k\bigg(\frac{m(q^2-1)(q+1)}{s}-1\bigg)+q+1},$$
which is not an integer because $k\geq1$. This contradiction completes the proof.
\end{itemize}
\end{proof}

The proof of Theorem \ref{mainY} is now complete. Indeed, by Proposition \ref{normalizzano} $\aut(\mathcal{Y}_{n,s})$ acts on $O \cup \{P_\infty\}$. From Lemma \ref{normaction} either $\aut(\mathcal{Y}_{n,s})=G$ or $\aut(\mathcal{Y}_{n,s})/C_{m/s}$ is isomorphic to ${\rm PGU}(3,q)$. From Lemma \ref{3divnlift} the latter case happens if and only if $3$ divides $n$ and $m/s$ divides $q^2-q+1$.

\subsection{The automorphism group of $\mathcal{X}_{a,b,n,s}$}
Consider the curve
$$
\mathcal{Y}_{n,s}:\begin{cases}
W^{m/s}=V^{q^2}-V \\
V^{q+1}=U^q+U\\
\end{cases}
$$
with coordinate functions $u,v,w$, and the automorphism group
$$
E_{\bar{q}}=\left\{(u,v,w)\mapsto\left(u+\frac{\lambda}{c},v,w\right)\colon\lambda\in\mathbb{F}_{\bar{q}}\right\}\subset{\rm Aut}(\cY_{n,s}),
$$
which is elementary abelian of order $\bar{q}$ and is contained in $S_{q^3}$. The aim of this section is to prove the following theorem.
\begin{theorem} \label{mainX}
Assume that $b<a$ or $q^2\nmid(\frac{m}{s}-1)$.
Then the automorphism group of $\cX_{a,b,n,s}$ has order $\frac{q^3}{\bar{q}}(q+1)(\bar{q}-1)\frac{m}{s}$ and is isomorphic to $(S_{q^3}/E_{\bar{q}})\rtimes C_{(q+1)(\bar{q}-1)m/s}$. 
\end{theorem}

\begin{remark}
Suppose that $b=a$, that is $\bar{q}=q$. Then $\mathcal{X}_{a,a,n,s}$ is the plane curve $Z^{m/s}=Y^{q^2}-Y$. If $q^2\nmid(\frac{m}{s}-1)$, then we can apply \cite[Theorem 3.2]{BMZsep}, whose case (ii) holds. This immediately yields the claim of Theorem \ref{mainX}.
Therefore, we can assume from now on that $b<a$.
\end{remark}

We start by noting that $\mathcal{X}_{a,b,n,s}$ is the quotient curve of $\mathcal{Y}_{n,s}$ over $E_{\bar q}$.
Indeed, define $x=cu-(cu)^{\bar{q}}$, $y=v$, $z=w$.
Then the function field of $\mathcal{X}_{a,b,n,s}$ is exactly $\mathbb{F}_{q^{2n}}(x,y,z)$; see \cite[Definition 3.1]{TTT}.
By direct computation $E_{\bar q}$ fixes $x,y,z$, i.e. $\mathbb{K}(\cX_{a,b,n,s})\subseteq\mathbb{K}(\cY_{n,s})^{E_{\bar q}}$. Since $|E_{\bar q}|=[\KK(\cY_{n,s}):\KK(\cX_{a,b,n,s})]$, equality holds and $\cX_{a,b,n,s}=\mathcal{Y}_{n,s}/E_{\bar q}$.

Therefore, an automorphism group $G$ of $\mathcal{X}_{a,b,n,s}$ is given by
\[
G:=\frac{\mathrm{N}_{{\rm Aut}(\cY_{n,s})}(E_{\bar q})}{E_{\bar q}}=\frac{S_{q^3}}{E_{\bar q}}\rtimes C_{(q+1)(\bar{q}-1)\frac{m}{s}}.
\]
The group $G$ has exactly one fixed point, namely $P_{\infty}$. We want to prove that $G$ is the whole stabilizer of $P_{\infty}$ in ${\rm Aut}(\cX_{a,b,n,s})$.

To this aim, we compute in the next propositions the contribution of the $p$-elements $\alpha\in G$ to the covering $\cX_{a,b,n,s}\to\cX_{a,b,n,s}/H$, where $H$ is any subgroup of ${\rm Aut}(\cX_{a,b,n,s})$ containing $\alpha$.

\begin{proposition}\label{contr1X}
Let $\alpha\in E_{q}/E_{\bar q}$. Then $i(\alpha)=(q^n+1)/s+1$.
\end{proposition}

\begin{proof}
Differently from the proof of Lemma \ref{contr1}, the non-trivial elements of $E_q/E_{\bar{q}}$ are not in a unique orbit under conjugation in $G$. However, they still give the same contribution. To see this, let $\alpha_1,\alpha_2\in (E_q/E_{\bar{q}})\setminus\{id\}$, and let $\beta_1,\beta_2\in E_q\setminus E_{\bar{q}}$ be such that $\alpha_1,\alpha_2$ are the cosets respectively of $\beta_1,\beta_2$ in $E_q/E_{\bar{q}}$.
Define $H_i=\langle \beta_i,E_{\bar q}\rangle=\langle \beta_i\rangle\times E_{\bar{q}}$, for $i=1,2$.
By Lemma \ref{contr1}, the curves $\cY_{n,s}/H_1$ and $\cY_{n,s}/H_2$ have the same genus. Clearly, $\cY_{n,s}/H_i\cong\cX_{a,b,n,s}/\langle\alpha_i\rangle$ for $i=1,2$.
Thus, $g(\cX_{a,b,n,s}/\langle\alpha_1\rangle)=g(\cX_{a,b,n,s}/\langle\alpha_2\rangle)$.
As $\langle\alpha_i\rangle$ is cyclic of prime order, this implies $i(\alpha_1)=i(\alpha_2)$.

Since $\alpha$ fixes $y$ and $z$, and the extension $\mathbb{F}_{q^{2n}}(x,y,z)/\mathbb{F}_{q^{2n}}(y,z)$ has degree $q/\bar{q}=|E_q/E_{\bar{q}}|$, the fixed field of $E_q/E_{\bar q}$ is exactly $\mathbb{F}_{q^{2n}}(y,z)$, whose genus is $g(\mathbb{F}_{q^{2n}}(y,z))=\frac{(m/s-1)(q^2-1)}{2}$.
Now the claim follows by applying the Hurwitz genus formula to $\mathbb{F}_{q^{2n}}(x,y,z)/\mathbb{F}_{q^{2n}}(y,z)$.
\end{proof}

\begin{proposition}\label{contr2X}
Let $\alpha\in (S_{q^3}/E_{\bar q})\setminus(E_q/E_{\bar q})$. Then $i(\alpha)=m/s+1$.
\end{proposition}

\begin{proof}
    The proof generalizes the one of Lemma \ref{contr2}. Since $z$ is fixed by $S_{q^3}/E_{\bar q}$ and $(z)_{\infty}=q^3/{\bar q}P_{\infty}$ by Equation \eqref{div z}, we have that the fixed field of $S_{q^3}/E_{\bar q}$ is $\mathbb{F}_{q^{2n}}(z)$ and hence is rational.
    As $E_{\bar q}$ is central in $S_{q^3}$ and $E_q$ is the only proper normal subgroup of $S_{q^3}$ containing $E_{\bar q}$, we have that $E_q/E_{\bar q}$ is the only proper normal subgroup of $S_{q^3}/E_{\bar q}$. Therefore, as in the proof of Lemma \ref{contr2}, the only possible non-trivial higher ramification groups $(S_{q^3}/E_{\bar q})^{(i)}$ are $S_{q^3}/E_{\bar q}$ and $E_q/E_{\bar q}$.
    Then the degree of the different divisor in $\mathbb{F}_{q^{2n}}(x,y,z)/\mathbb{F}_{q^{2n}}(z)$ is $(n-j)(q/\bar{q}-1)+j(q^3/\bar{q}-1)$, where $j$ is the number of ramification groups coinciding with $S_{q^3}/E_{\bar q}$ and $n=(q^n+1)/s+1$ by Lemma \ref{contr1X}.
    From the Hurwitz genus formula applied to $\mathbb{F}_{q^{2n}}(x,y,z)/\mathbb{F}_{q^{2n}}(z)$, it follows that $j=m/s+1$. The claim follows.
\end{proof}

By Lemma \ref{highnormal}, we can write ${\rm Aut}(\cX_{a,b,n,s})_{P_{\infty}}=\tilde{S}\rtimes\tilde{C}$, where $\tilde{S}$ is the Sylow $p$-subgroup of ${\rm Aut}(\cX_{a,b,n,s})$ and $\tilde{C}$ is a cyclic $p^\prime$-group.
Clearly $S_{q^3}/E_{\bar q}\leq \tilde{S}$ and, up to conjugation, we can assume $C_{(q+1)(\bar{q}-1)m/s}\leq \tilde{C}$. 

\begin{lemma}\label{lem:complX}
The equality $\tilde{C}=C_{(q+1)(\bar{q}-1)m/s}$ holds.
\end{lemma}

\begin{proof}
From $C_{m/s}\triangleleft \tilde{C}$ it follows that $\tilde{C}/C_{m/s}$ is an automorphism group of the fixed field $\mathbb{F}_{q^{2n}}(x,y)$ of $C_{m/s}$, of order at least $|C_{(q+1)(\bar{q}-1)\frac{m}{s}}/C_{m/s}|=(q+1)(\bar{q}-1)$.
By \cite[Theorem 3.3]{BMZsep}, the $p^\prime$-part of $|{\rm Aut}(\mathbb{F}_{q^{2n}}(x,y))_{P_{\infty}}|$ is at most $(q+1)(\bar{q}-1)$.
Then equality holds and the claim follows.
\end{proof}

\begin{lemma}\label{lem:pX}
The equality $\tilde{S}=S_{q^3}/E_{\bar q}$ holds.
\end{lemma}

\begin{proof}
The first step is to prove that $i(\sigma)=2$ for all $\sigma\in\tilde{S}\setminus( S_{q^3}/E_{\bar q})$.
Since $i(\sigma)\geq2$ for all $\sigma\in\tilde{S}$ and $i(\alpha)$ has been computed in Propositions \ref{contr1X}, \ref{contr2X} for all $\alpha\in S_{q^3}/E_{\bar{q}}$, the claim $i(\sigma)=2$ follows by direct computation from the Hurwitz genus formula, in analogy with the proof of Lemma \ref{normpsyl}.
Therefore $i(\alpha)\geq i(\beta)$ for any non-trivial $\alpha\in E_q/E_{\bar q}$ and $\beta\in {\rm Aut}(\cX_{a,b,n,s})_{P_{\infty}}$, so that $E_q/E_{\bar q}$ is the last non-trivial higher ramification group at $P_{\infty}$. By Lemma \ref{highnormal}, this implies that $E_q/E_{\bar q}$ is normal in $\tilde{S}$.
Therefore $\tilde{S}/(E_q/E_{\bar q})$ is an automorphism group of the fixed field of $E_q/E_{\bar q}$, which clearly coincides with $\mathbb{F}_{q^{2n}}(y,z)$ where $z^{m/s}=y^{q^2}-y$.
By \cite[Theorem 3.2]{BMZsep}, a Sylow $p$-subgroup of the automorphism group of $\mathbb{F}_{q^{2n}}(y,z)$ has order $q^2$. This is equal to the size of $(S_{q^3}/E_{\bar q})/(E_q/E_{\bar q})$, and the claim follows.
\end{proof}

Lemmas \ref{lem:complX} and \ref{lem:pX} complete the proof of the following result.

\begin{corollary}\label{cor:fixX}
The full stabilizer of $P_{\infty}$ in ${\rm Aut}(\cX_{a,b,n,s})$ is
$$G=(S_{q^3}/E_{\bar q})\rtimes C_{(q+1)(\bar{q}-1)m/s}.$$
\end{corollary}

Define $O=\left\{ 
P_{(\alpha,\beta,0)}\mid \beta\in\mathbb{F}_{q^2},\, c\beta^{q+1}=\mathrm{Tr}_{q/\bar{q}}(\alpha) \right\}$, and consider the set $O\cup\{P_{\infty}\}\subset\cX_{a,b,n,s}(\mathbb{F}_{q^{2n}})$, which is stabilized pointwise by $C_{m/s}$.

\begin{lemma}\label{lem:normX}
The pointwise stabilzer of $O\cup\{P_{\infty}\}$ in ${\rm Aut}(\cX_{a,b,n,s})$ is $C_{m/s}$.
\end{lemma}

\begin{proof}
If $\alpha$ stabilizes $O\cup\{P_{\infty}\}$ pointwise, then $\alpha$ preserves the principal divisors of the coordinate functions $x,y,z$. Thus $\alpha:(x,y,z)\mapsto(\lambda x,\mu y, \rho z)$ for some $\lambda,\mu,\rho\in\mathbb{F}_{q^{2n}}$.
By direct checking with the equations of $\cX_{a,b,n,s}$, this implies $\lambda=\mu=1$ and $\rho^{m/s}=1$, that is $\alpha\in C_{m/s}$. 
\end{proof}

\begin{corollary}\label{cor:actX}
If ${\rm Aut}(\cX_{a,b,n,s})$ acts on $O\cup\{P_{\infty}\}$, then ${\rm Aut}(\cX_{a,b,n,s})={\rm Aut}(\cX_{a,b,n,s})_{P_{\infty}}$.
\end{corollary}

\begin{proof}
By Lemma \ref{lem:normX}, $C_{m/s}$ is normal in ${\rm Aut}(\cX_{a,b,n,s})$. Then ${\rm Aut}(\cX_{a,b,n,s})/C_{m/s}$ is an automorphism group of the fixed field $\mathbb{F}_{q^{2n}}(x,y)$ of $C_{m/s}$. By \cite[Theorem 12.11]{HKT} (see also \cite[Lemma 2.3]{BMZsep}), ${\rm Aut}(\cX_{a,b,n,s})/C_{m/s}$ stabilizes the unique point at infinity of $\mathbb{F}_{q^{2n}}(x,y)$, which is totally ramified under $P_{\infty}$. Since $C_{m/s}$ fixes $P_{\infty}$, this implies that ${\rm Aut}(\cX_{a,b,n,s})$ fixes $P_{\infty}$.
\end{proof}

By Corollaries \ref{cor:fixX} and \ref{cor:actX}, Theorem \ref{mainX} is proved once we show that ${\rm Aut}(\cX_{a,b,n,s})$ acts on $O\cup\{P_{\infty}\}$.
By contradiction, assume from now on that this is not the case.

Suppose first that no point in $\cX_{a,b,n,s}\setminus(O\cup\{P_{\infty}\})$ is in the same orbit $O_{\infty}$ of $P_{\infty}$ under ${\rm Aut}(\cX_{a,b,n,s})$.
Since ${\rm Aut}(\cX_{a,b,n,s})$ does not act on $O\cup\{P_{\infty}\}$, there exist two points $P\in\cX_{a,b,n,s}\setminus(O\cup\{P_{\infty}\})$ and $Q\in O$ lying in the same orbit $O_P$. Then $O\subseteq O_P$, because $O$ is an orbit under $G$. As $P_\infty\notin O_P$, this implies that $O_{\infty}=\{P_\infty\}$; in this case, the claim follows from Corollary \ref{cor:fixX}.

Therefore there exists $P\in\cX_{a,b,n,s}\setminus(O\cup\{P_{\infty}\})$ with $P\in O_{\infty}$.
Also, $P$ is $\mathbb{F}_{q^{2n}}$-rational, because the automorphism group of the $\mathbb{F}_{q^{2n}}$-maximal curve $\cX_{a,b,n,s}$ is defined over $\mathbb{F}_{q^{2n}}$, and $P$ lies in the same orbit of the $\mathbb{F}_{q^{2n}}$-rational point $P_{\infty}$.

\begin{lemma}
The short orbit $O_{\infty}$ is the only non-tame orbit of ${\rm Aut}(\cX_{a,b,n,s})$.
\end{lemma}

\begin{proof}
The proof is analogous to the one of Lemma \ref{onenontame}. It relies on the fact that all Sylow $p$-subgroups are conjugate, together with the correspondence between a point in a non-tame orbit and the Sylow $p$-subgroup made by the $p$-elements fixing that point.
\end{proof}

Write
\begin{equation}\label{eq:sizeX}|O_{\infty}|=1+\ell|O|+k|{\rm Aut}(\cX_{a,b,n,s})_{P_{\infty}}|=1+i\frac{q^3}{\bar q}+k\frac{q^3}{\bar q}(q+1)(\bar{q}-1)\frac{m}{s},\end{equation}
where $\ell$ is $1$ or $0$ according to $O\subset\cO_{\infty}$ or $O\not\subset O_{\infty}$, and $k\geq1$ is the number of long orbits of ${\rm Aut}(\cX_{a,b,n,s})_{P_{\infty}}$ contained in $O_{\infty}$.
Since $k\geq1$, the orbit-stabilizer theorem yields
$$
|{\rm Aut}(\cX_{a,b,n,s})|=|O_{\infty}|\cdot|G|>|G|^2>84(g-1).
$$
Then, by Theorem \ref{Hurwitz}, one of the following cases holds for the short orbits of ${\rm Aut}(\cX_{a,b,n,s})$:
\begin{itemize}
    \item[(A)] exactly one short orbit $O_{\infty}$, non-tame, of length dividing $2g(\cX_{a,b,n,s})-2$;
    \item[(B)] exactly one non-tame orbit $O_{\infty}$ and two tame orbits, both of length ${|{\rm Aut}(\cX_{a,b,n,s})|}{2}$, with $p\geq3$;
    \item[(C)] exactly one non-tame orbit $O_{\infty}$ and one tame orbit.
\end{itemize}
The case (A) cannot occur, because $k\geq1$ implies $|O_{\infty}|>2g(\cX_{a,b,n,s})-2$.
In the next lemmas we find a contradiction also to the cases (B) and (C).

\begin{lemma}
The case (B) does not occur.
\end{lemma}

\begin{proof}
Suppose that the case (B) occurs. In analogy with Lemma \ref{numborbits}, we apply the Hurwitz genus formula to ${\rm Aut}(\cX_{a,b,n,s})$. By Lemmas \ref{contr1X} and \ref{contr2X}, we obtain
$$
2g(\cX_{a,b,n,s})-2=-2|{\rm Aut}(\cX_{a,b,n,s})|$$
$$+2\frac{|{\rm Aut}(\cX_{a,b,n,s})|}{2}(2-1)+|O_{\infty}|\left(|G|-1+\left(\frac{q}{\bar{q}}-1\right)(q+1)\frac{m}{s}+\left(\frac{q^3}{\bar{q}}-\frac{q}{\bar{q}}\right)\frac{m}{s}\right).
$$
Since $|O_{\infty}|\cdot|G|=|{\rm Aut}(\cX_{a,b,n,s})|$, we get
$$
\left(\frac{q^2}{\bar{q}}-1\right)(q+1)\frac{m}{s}-\left(\frac{q^3}{\bar{q}}+1\right)=\left(1+\ell\frac{q^3}{\bar{q}}+k|G|\right)\cdot\left(\left(\frac{q^2}{\bar{q}}-1\right)(q+1)\frac{m}{s}-1\right).
$$
Since $k\geq1$, this is a contradiction.
\end{proof}

\begin{lemma}
The case (C) does not occur.
\end{lemma}

\begin{proof}
    The proof is analogous to the one of Proposition \ref{normalizzano}.
    Suppose that the case (C) occurs, and let $O_1$ be the tame short orbit. Then $O$ is contained in $O_{\infty}$ or $\cO_1$ according to $\ell=1$ or $\ell=0$. Thus, $O_1$ has size $(1-\ell)|O|+k_1|G|$, where $k_1\geq0$ is the number of long orbits of $G$ contained in $O_1$.
    By the Hurwitz genus formula we obtain
    $$
    2g(\cX_{a,b,n,s})-2=$$
    $$-2|{\rm Aut}(\cX_{a,b,n,s})|+|O_{\infty}|\left(|G|-1+\left(\frac{q^2}{\bar{q}}-1\right)(q+1)\frac{m}{s}\right)+|O_1|\left(\frac{|{\rm Aut}(\cX_{a,b,n,s})|}{|O_1|}-1\right),
    $$
    which yields
    \begin{equation}\label{eq:k1X}
    k_1=k\left(\left(\frac{q^2}{\bar{q}}-1\right)(q+1)\frac{m}{s}-1\right)+\ell\left(\frac{q^2/\bar{q}-1}{\bar{q}-1}\right).
    \end{equation}
    By the orbit-stabilizer theorem,
    \begin{equation}\label{eq:OSTX}
    |{\rm Aut}(\cX_{a,b,n,s})_{P_1}|\cdot|O_1|=|G|\cdot|O_{\infty}|,
    \end{equation}
    where $P_1$ is a point of $O_1$. We analyze the cases $\ell=0$ and $\ell=1$ separately.
    \begin{itemize}
        \item Let $\ell=0$. Then $O\subseteq O_1$ and ${\rm Aut}(\cX_{a,b,n,s})_{P_1}$ has a subgroup conjugate to $C_{(q+1)(\bar{q}-1)m/s}$. We can then write $|{\rm Aut}(\cX_{a,b,n,s})_{P_1}|=h(q+1)(\bar{q}-1)m/s$ for some $h\geq1$.
        By direct computation with Equations \eqref{eq:k1X} and \eqref{eq:OSTX}, we get
        \begin{equation}\label{eq:accaX}
        h\left(1+k(q+1)(\bar{q}-1)\frac{m}{s}\left(\left(\frac{q^2}{\bar{q}}-1\right)(q+1)\frac{m}{s}-1\right)\right)=1+k(q+1)(\bar{q}-1)\frac{m}{s}.
        \end{equation}
        In particular, $h\equiv1\pmod{(q+1)(\bar{q}-1)\frac{m}{s}}$.

        Suppose that $h>1$. Then $h>(q+1)(\bar{q}-1)\frac{m}{s}$, so that
        $$
        |{\rm Aut}(\cX_{a,b,n,s})_{P_1}|>\left((q+1)(\bar{q}-1)\frac{m}{s}\right)^2
        $$
        and hence
        $$
        |{\rm Aut}(\cX_{a,b,n,s})|=|{\rm Aut}(\cX_{a,b,n,s})_{P_1}|\cdot|O_1|>$$
        $$\left((q+1)(\bar{q}-1)\frac{m}{s}\right)^2\cdot\left(\frac{q^3}{\bar{q}}+k\left(\left(\frac{q^2}{\bar{q}}-1\right)(q+1)\frac{m}{s}-1\right)\frac{q^3}{\bar{q}}(q+1)(\bar{q}-1)\frac{m}{s}\right).
        $$
        Since $k>0$, this implies
        $$
        |{\rm Aut}(\cX_{a,b,n,s})|>8g^3,
        $$
        a contradiction to Theorem \ref{Hurwitz}. Therefore $h=1$. Now Equation \eqref{eq:accaX}, together with $h=1$ and $k\geq1$, provides a contradiction.
        \item Let $\ell=1$. By direct computation using Equations \eqref{eq:k1X} and \eqref{eq:OSTX}, one gets
        \begin{equation}\label{eq:noninteroX}
        |{\rm Aut}(\cX_{a,b,n,s})_{P_1}|=q\bar{q}-q+\frac{ kq(q+1)(\bar{q}-1)\frac{m}{s}-\left(\frac{q^3}{\bar{q}}-q\bar{q}-1\right) }{ k\left((q+1)\left(\frac{q^2}{\bar{q}}-1\right)\frac{m}{s}-1\right)+\frac{q^2/\bar{q}-1}{\bar{q}-1} }.
        \end{equation}
        Consider the fraction on the right-hand side of Equation \eqref{eq:noninteroX}, and recall $k\geq1$. The denominator is positive and greater than the absolute value of the numerator. Also, the numerator is non-zero, being congruent to $1$ modulo $q$. Thus, the right-hand side of Equation \eqref{eq:noninteroX} is not an integer, a contradiction. 
    \end{itemize}
\end{proof}
The proof of Theorem \ref{mainX} is now complete.

\section*{Acknowledgments}
The first author would like to acknowledge the support from The Danish Council for Independent Research (DFF-FNU) for the project \emph{Correcting on a Curve}, Grant No.~8021-00030B.
The third author was partially supported by the Italian National Group for Algebraic and Geometric Structures and their Applications (GNSAGA - INdAM).


\begin{thebibliography}{99}
\bibitem{zini} D. Bartoli, M. Montanucci and G. Zini, {\em Multi point AG codes on the GK maximal curve}, Designs, Codes and Cryptography, 86, pp. 161–177 (2018).

\bibitem{BMZ} D. Bartoli, M. Montanucci, and G. Zini, \emph{AG codes and AG quantum codes from the GGS curve}, Des. Codes Cryptogr. {\bf 86} (2018), no. 10, 2315--2344.

\bibitem{BM}  P. Beelen and M. Montanucci. A new family of maximal curves. \textit{Journal of the London Math. Soc.} \textbf{2} (2018), 1--20.

\bibitem{BMZsep}  M. Bonini, M. Montanucci and G. Zini, On plane curves given by separated polynomials and their automorphisms. \textit{Adv. Geom.} \textbf{20} (2020), no. 1, 61--70.

\bibitem{CT_GK} A. S. Castellanos, G. Tizziotti, {\em Two-Point AG Codes on the GK Maximal Curves}. IEEE Transactions on Information Theory, v. 62, p. 681-686, 2016.

\bibitem{DM} I. M. Duursma and K.-H. Mak, \emph{On maximal curves which are not Galois subcovers of the Hermitian curve}, Bull. Braz. Math. Soc (N.S.) {\bf 43 (3)} (2012), 453--465.

\bibitem{GKcodes} S. Fanali and M. Giulietti, {\em One-point AG Codes on the GK Maximal Curves}, IEEE Trans. on Information Theory, vol. 56, no. 1, pp. 202 - 210, 2010.

\bibitem{GGS} A. Garcia, C. G\" uneri, and H. Stichtenoth, \emph{A generalization of the Giulietti-Korchm\' aros maximal curve}, Adv. Geom. 10 (2010), no. 3, 427--434.

\bibitem{GK}  M. Giulietti and G. Korchm\'aros, \emph{A new family of maximal curves over a finite field}, Math. Ann. 343 (2009), 229--245.

\bibitem{GMZ} M. Giulietti, M. Montanucci and G. Zini, {\it On maximal curves that are not quotients of the Hermitian curve}, Finite Fields Appl. {\bf 41} (2016), 72--88.

\bibitem{GSY} B. Gunby, A. D. Smith and  A. Yuan, {\it Irreducible canonical representations in positive characteristic}, Res. Number Theory 1, 3 (2015). https://doi.org/10.1007/s40993-015-0004-8

\bibitem{GOS} C. G\"uneri, M. \"Ozdemir and H. Stichtenoth, {\it The automorphism group of the generalized Giulietti–Korchmáros function field} , Adv. Geom. vol. 13, no. 2, 2013, pp. 369-380.

\bibitem{GMP} R. Guralnick, B. Malmskog, R. Pries,{ \it The automorphism group of a family of maximal curves}, J. Algebra 361 (2012), 92–116.

\bibitem{henn1978} H.W. Henn,  Funktionenk\"orper mit gro{\ss}er Automorphismengruppe, {\em J. Reine Angew. Math.} \textbf{302} (1978), 96--115.

\bibitem{HKT} J.W.P. Hirschfeld, G. Korchm\'aros, and F. Torres, {\it Algebraic Curves over a Finite Field}, Princeton Series in Applied Mathematics, Princeton (2008).

\bibitem{HY3} C. Hu and S. Yang, \emph{Multi-point Codes from the GGS Curves}, Advances in Mathematics of Communications, to appear.

\bibitem{KS} S.L. Kleiman.\textit{Algebraic cycles and the Weil conjectures}, in: Dix expos\'es sur la cohomologie des sch\'emas, in: Adv. Stud. Pure Math. \textbf{3}, (1968) 359--386 .

\bibitem{Serre} G. Lachaud, {\it Sommes d' Eisenstein et nombre de points de certaines courbes alg\'ebriques sur les corps finis}, C.R. Acad. Sci. Paris 305 (Serie I) (1987), 729--732.

\bibitem{LV} L. Landi and L. Vicino, {\it Two-point AG codes from the Beelen-Montanucci maximal curve},  Finite Fields Appl. \textbf{80} (2022), Paper No. 102009, 17 pp.

\bibitem{MP} M. Montanucci and V. Pallozzi Lavorante, {\it AG codes from the second generalization of the GK maximal curve}, Discrete Math. 343 (2020), 111810.

\bibitem{MT} M. Montanucci and G. Tizziotti, {\em Generalized Weierstrass semigroups at several points on certain maximal curves which cannot be covered by the Hermitian curve}, Des. Codes Cryptogr. (2022). https://doi.org/10.1007/s10623-022-01130-3.

\bibitem{RS} H.G. R\"uck and H. Stichtenoth. A characterization of the Hermitian function fields over finite fields. \emph{J. Reine Angew. Math.} {\bf 457} (1994), 185--188.

\bibitem{serre1979} J.P.~Serre, \emph{Local Fields}, Graduate Texts in Mathematics {\bf 67}, Springer, New York, 1979, viii+241 pp.

\bibitem{stichtenoth} H. Stichtenoth, {\em Algebraic Function Fields and Codes}, Berlin, Germany: Springer, 1993.

\bibitem{stichtenoth1973II} H.~Stichtenoth, \"Uber die Automorphismengruppe eines algebraischen Funktionenk{\"o}rpers von Primzahl- charakteristik. II. Ein spezieller Typ von Funktionenk{\"o}rpern, \emph{Arch. Math.} \textbf{24} (1973), 615--631.

\bibitem{sullivan1975} F.~Sullivan,  $p$-torsion in the class group of curves with many automorphisms, \emph{Arch. Math.} {\bf 26} (1975), 253--261.

\bibitem{TTT} S. Tafazolian, A. Teher\'an-Herrera, and F. Torres, {\em Further examples of
maximal curves which cannot be covered by the Hermitian curve}. J. Pure Appl. Algebra,
220(3) : 1122--1132, 2016.

\bibitem{CT mGK} G. Tizziotti and A. S. Castellanos, {\em Weierstrass Semigroup and Pure Gaps at Several Points on the GK Curve}. Bulletin Brazilian Mathematical Society, v. 49, p. 419--429, 2018. 

\bibitem{VillaSalvador} G. D. Villa Salvador, Topics in the theory of algebraic function fields, Mathematics: Theory and Applications. Birkh\"auser Boston, Inc., Boston, MA, (2006).

\end{thebibliography}
    \end{document}